\documentclass{siamart0516}

\usepackage{amssymb,amsfonts,amsmath}
\usepackage[mathscr]{eucal}
\usepackage{bm}
\usepackage{stmaryrd}
\usepackage{braket}
\usepackage{xfrac} %% For \sfrac
\usepackage{relsize}
\usepackage{xspace}
\usepackage{xcolor}

\usepackage{graphicx}
\graphicspath{{./},{./}}
\usepackage[caption=false]{subfig}
\usepackage{array}
\usepackage{siunitx}

\usepackage{algpseudocode}
\usepackage{cleveref}

\usepackage{pgfplots}
\usepackage{pgfplotstable}
\usepgfplotslibrary{statistics}

\definecolor{cpalsncolor}{HTML}{A6611A}
\definecolor{cpalsrcolor}{HTML}{DFC27D}
\definecolor{cprandcolor}{HTML}{80CDC1}
\definecolor{cprandfcolor}{HTML}{018571}

\pgfplotsset{every axis plot/.append style={line width=1pt}}

\newcommand{\bigO}[1]{\ensuremath{\mathcal{O}\hspace{-.2em}\left({#1}\right)}}
\newcommand{\inlineBigO}[1]{\ensuremath{\mathcal{O}({#1})}}
\newcommand{\T}[2][]{\boldsymbol{#1\mathscr{\MakeUppercase{#2}}}}
\newcommand{\M}[1]{\ensuremath{\mathbf{\uppercase{#1}}}\xspace} % matrix
\newcommand{\Mn}[2]{\M{#1}^{(#2)}\xspace} % matrix in sequence
\newcommand{\MnC}[3]{\V{#1}^{(#2)}_{#3}} % Column of matrix in sequence
\newcommand{\V}[1]{\ensuremath{\mathbf{\lowercase{#1}}}\xspace} % vector
\newcommand{\R}{\ensuremath{\mathbb{R}}\xspace}

\newcommand{\trans}{\ensuremath{\mathsf{T}}}

\newcommand{\qtext}[1]{\quad\text{#1}\quad}
\usepackage{booktabs}

\newcommand{\LineRef}[1]{\hyperref[#1]{line~\ref{#1}}}

\usepackage[normalem]{ulem}
\DeclareMathOperator*{\argmin}{arg\,min}

\Crefname{ALC@unique}{Line}{Lines}

\newcommand{\modeidx}{\ensuremath{m}}
\newcommand{\samplesize}{\ensuremath{S}}

\pgfplotsset{every axis/.append style={
    label style={font=\footnotesize},
    tick label style={font=\footnotesize},  
    every axis plot post/.style={
          mark=*,
          every mark/.style={mark size=1pt}}
  }}

\begin{document}
\title{A Practical Randomized CP Tensor Decomposition%
  \thanks{This material is based upon work supported by the U.S. Department of Energy, Office of Science, Office of Advanced Scientific Computing Research, Applied Mathematics program.
    Sandia National Laboratories is a multimission laboratory managed and operated by National Technology and Engineering Solutions of Sandia, LLC., a wholly owned subsidiary of Honeywell International, Inc., for the U.S. Department of Energy's National Nuclear Security Administration under contract DE-NA-0003525.}}
\author{
  Casey Battaglino\thanks{Georgia Institute of Technology Computational Sci. and Engr. (\email{cbattaglino3@gatech.edu}).} \and
  Grey Ballard\thanks{Wake Forest University (\email{ballard@wfu.edu}).} \and
  Tamara G. Kolda\thanks{Sandia National Laboratories (\email{tgkolda@sandia.gov}).}
}

\maketitle

\begin{abstract}
  The CANDECOMP/PARAFAC (CP) decomposition is a leading method for the analysis of multiway data. 
The standard alternating least squares algorithm for the CP decomposition (CP-ALS) involves a series of highly overdetermined linear least squares problems. 
We extend randomized least squares methods to tensors and show the
workload of CP-ALS can be drastically reduced without a sacrifice in
quality.
We introduce techniques for efficiently preprocessing, sampling, and
computing randomized least squares on a dense tensor of arbitrary
order, as well as an efficient sampling-based technique for checking
the stopping condition.
We also show more generally that the Khatri-Rao product (used within
the CP-ALS iteration) produces conditions favorable for direct
sampling. 
In numerical results, we see improvements in speed, reductions in
memory requirements, and robustness with respect to initialization. 
\end{abstract}

% REQUIRED

% REQUIRED

\section{Introduction} \label{sec:intro} 
The CANDECOMP/PARAFAC (CP) tensor decomposition is an important tool
for data analysis in applications such as chemometrics~\cite{MuStGrBr13}, biogeochemistry~\cite{JaCaYa14},
neuroscience~\cite{AcBiBiBr07,DaGiCaWa13,CoLiKuGo15}, signal processing~\cite{SiDeFuHu16}, cyber traffic analysis~\cite{MaGuFa11}, and many others.
We consider the problem of accelerating the alternating least squares (CP-ALS) algorithm using randomization.

Because randomized methods have been used successfully for solving
linear least squares problems~\cite{DrMaMuSa11,blendenpik,sketching}, it is natural that they might prove
beneficial to CP-ALS since its key kernel is the solution of a least
squares problem. However, the CP-ALS least squares subproblem has a
special structure that already greatly reduces its cost, so it is unclear whether or not sketching would be beneficial.
Nevertheless, we find that our randomized algorithms significantly reduce the memory and computational overhead of the CP-ALS process for dense tensors and moreover positively impact algorithmic robustness.
To the best of our knowledge, this is the first successful application of matrix sketching methods in the context of CP.
The contributions of this paper are as follows:
\begin{itemize}
\item 
The least squares coefficient matrix in the CP-ALS subproblem is a Khatri-Rao product of factor matrices. 
Our randomized algorithm prefers \emph{incoherent} matrices.
We prove that the coherence of the Khatri-Rao product is bounded 
above by the product of the coherence of its factors.
\item We introduce the CPRAND algorithm that uses a randomized least squares solver for the subproblems in CP-ALS and never explicitly forms the full Khatri-Rao matrices used in the subproblems.
We also introduce the complementary CPRAND-MIX algorithm that employs efficient \emph{mixing} to promote incoherence and thereby improves the robustness of the method.
\item We derive a novel, lightweight stopping condition that estimates
  the model fit error, and we prove its accuracy using Chernoff-Hoeffding bounds.
\item We demonstrate the speed and robustness of our algorithms over a
  large number of synthetic tensors as well as real-world data
  sets. In comparison with CP-ALS, CPRAND is faster and much less sensitive to
  the starting point. 
\end{itemize}
We give an example of our methods' fast time to solution in \cref{fig:tvsfmain},
comparing CPRAND and CPRAND-MIX with CP-ALS.
For the CPRAND methods, we use 100 sampled rows for each least squares solve.
The randomized methods converge much more quickly, in only a few
iterations. 
The fit is not monotonically increasing for the randomized methods due
to (small) variations in the solution to each randomized subproblem.
See~\cref{sec:experiments} for full details on problem generation and
further experiments.

\begin{figure}[tbhp]
  \centering \subfloat[Random $300\times300\times300$ tensor]{\label{fig:tvsf:3d}
    \begin{tikzpicture}
      \begin{axis}[width=.49\textwidth, height=2.1in, grid=major,
        xlabel={time ($s$)}, ylabel={fit},
        xmin=0,xmax=20,ymin=0.9,ymax=1,legend style={font=\smaller},
        legend pos=south east,style={thick}
        ]
        \addplot[color=cpalsncolor,mark=*,mark size=1pt] table [x=cpt,y=cpf] {./tvsf300.dat};
        \addplot[color=cprandcolor,mark=*,mark size=1pt] table [x=cprandt,y=cprandf] {./tvsf300.dat};
        \addplot[color=cprandfcolor,mark=*,mark size=1pt] table [x=cprandfftt,y=cprandfftf] {./tvsf300.dat};
        \addplot[mark=none, black, samples=2,very thin,dashed] coordinates {(0,0.99) (20,0.99)};
      \end{axis}
    \end{tikzpicture}
    \label{fig:tvsf}}
  \subfloat[Random $80\times80\times80\times80$ tensor]{\label{fig:tvsf:4d}
    \begin{tikzpicture}
      \begin{axis}[width=.49\textwidth, height=2.1in, grid=major,
        xlabel={time ($s$)},
        xmin=0,xmax=20,ymin=0.9,ymax=1,legend style={font=\smaller},
        legend entries={CP-ALS,CPRAND,CPRAND-MIX}, legend pos=south east,style={thick}
        ]
        \addplot[color=cpalsncolor,mark size=1pt,mark=*] table [x=cpt,y=cpf] {./tvsf80.dat};
        \addplot[color=cprandcolor,mark size=1pt,mark=*] table [x=cprandt,y=cprandf] {./tvsf80.dat};
        \addplot[color=cprandfcolor,mark size=1pt,mark=*] table [x=cprandfftt,y=cprandfftf] {./tvsf80.dat};
        \addplot[mark=none, black, samples=2,very thin,dashed] coordinates {(0,0.99) (20,0.99)};
      \end{axis}
    \end{tikzpicture}
    \label{fig:tvsf2}}
  \label{fig:tvsfmain}
  \caption{Runtime comparison for fitting the CP tensor decomposition
    on random synthetic tensors generated
    to have rank 5, factor collinearity of 0.9, and 1\% noise.
    We compare a single run of three methods using a target rank of 5.
    CPRAND and CPRAND-MIX use random initialization,
    100 sampled rows for each least squares solve. 
    CP-ALS uses HOSVD initialization.
    The marks indicate each iteration.
    The
    thin dashed black line represents a fit of 99\%, which is the best
    we expect when the noise is 1\%.} 
  \end{figure}
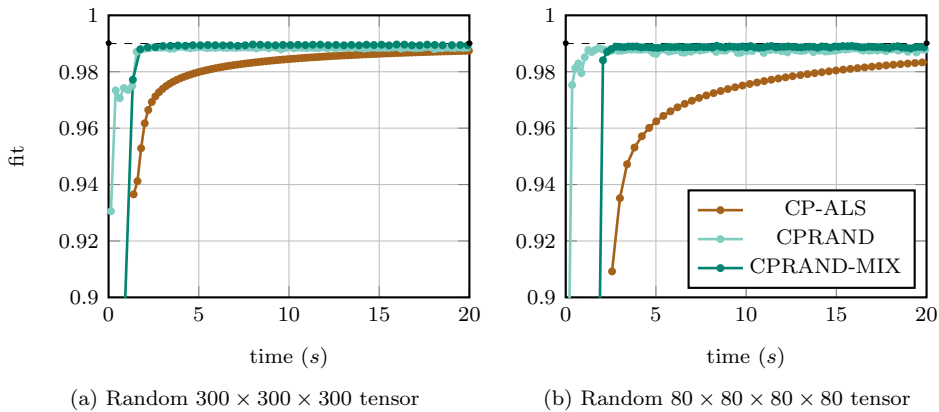

\section{Background and Definitions} \label{sec:background} 
In this section we provide information on key matrix and tensor operations, as well as randomized least squares.

\subsection{Matrix and Tensor Background}
A tensor is an element in a tensor product of one or more vector spaces. In data analysis it suffices to think about a tensor as a multidimensional array.
We represent a tensor as a Euler script capital letter, e.g., $\T{X}
\in \R^{I_1 \times \cdots \times I_N}$.  
The number of \emph{modes} (or dimensions) is referred to as the
\emph{order}, denoted here by $N$.  
The \emph{mode-$n$ fibers} of a tensor are the higher-order analogue
of matrix column and row vectors.  
The mode-$n$ \emph{unfolding} or \emph{matricization} of a tensor 
aligns the mode-$n$ fibers as the columns of an $I_n \times \prod_{m
  \neq n} I_m$ matrix. 
Assuming 1-indexing, tensor entry $x_{i_1,i_2,\dots,i_N}$ then maps to
entry $(i_n, j)$ of $\M{X}_{(n)}$ via the relation: 
\begin{equation}
\label{eqn:unfoldmapping}
j = 1+\sum_{\substack{k=1\\k\neq n}}^N (i_k -1)J_k, \qquad \text{where} \qquad J_k = \prod_{\substack{m=1\\m\neq n}}^{k-1} I_m.
\end{equation}

Given matrices $\M{A} \in \mathbb{R}^{I \times J}$ and $\M{B} \in \mathbb{R}^{K
  \times L}$, their \emph{Kronecker product} is 
\begin{displaymath}
  \M{A} \otimes \M{B} =
  \begin{bmatrix}
    a_{11} \M{B} & a_{12} \M{B} & \cdots & a_{1J} \M{B} \\
    \vdots & \vdots & \ddots & \vdots \\
    a_{I1} \M{B} & a_{I2} \M{B} & \cdots & a_{IJ} \M{B} \\
  \end{bmatrix}
  \in \R^{IK \times JL}.
\end{displaymath}
Assuming $K=L$, their \emph{Khatri-Rao product}, also known as the
\emph{matching columnwise Kronecker product}, is
\begin{displaymath}
  \M{A} \odot \M{B} = 
  \begin{bmatrix}
    \V{a}_1 \otimes \V{b}_1 & \V{a}_2 \otimes \V{b}_2
    & \cdots &  \V{a}_K \otimes \V{b}_K .   
  \end{bmatrix}
\end{displaymath}
Assuming $I=K$ and $J=L$, their \emph{Hadamard product} is
$\M{A}\circledast\M{B} \in \mathbb{R}^{I \times J}$, the elementwise
product of the matrices.
Three useful identities  are:
\begin{align}
(\M{A} \odot \M{B})^\trans(\M{A} \odot \M{B}) &= \M{A}^\trans\M{A} \circledast \M{B}^\trans\M{B}, \label{eq:KRGram} \\ 
\M{A}\M{B} \otimes \M{C}\M{D} &= (\M{A} \otimes \M{C})(\M{B} \otimes \M{D}) \label{eqn:krondist},\quad\text{and} \\
\M{A}\M{B} \odot \M{C}\M{D} &= (\M{A} \otimes \M{C})(\M{B} \odot \M{D}). \label{eqn:krdist}
\end{align}

The \emph{mode-$n$ tensor-times-matrix product} is a contraction
between a matrix and a tensor in its $n$th mode. The two operations
below are equivalent and can be computed in place (i.e., without
explicitly unfolding $\M{X}_{(n)}$)~\cite{intensli}:
\[\T{Y} = \T{X} \times_n \M{A} \quad \Leftrightarrow \quad \M{Y}_{(n)} = \M{A}\M{X}_{(n)}.\]
If a tensor is followed by a series of mode-$n$ products, its mode-$n$
matricization has a particular structure~\cite{Kolda:2009} that will
useful in the discussion of ``mixing'' in \cref{sec:cprand-mix}: 
\begin{align}
\begin{split}
\label{eqn:ttensor}\T{Y} & ={}  \T{X} \times_1 \Mn{U}{1} \cdots \times_N \Mn{U}{N} \quad \Leftrightarrow \\
\Mn{Y}{n} & ={}  \Mn{U}{n} \M{X}_{(n)}(\Mn{U}{N} \otimes \cdots \otimes \Mn{U}{n+1} \otimes \Mn{U}{n-1} \otimes \cdots \otimes \Mn{U}{1})^\trans.
\end{split}
\end{align}

The CP tensor decomposition aims to approximate an order-$N$ tensor as
a sum of $R$ rank-one
tensors~\cite{hitchcock-sum-1927, CANDECOMP, PARAFAC, Kolda:2009}:  
\begin{equation}
\label{eqn:cpform}
\T{X} \approx \T{\tilde{X}} = \sum_{r=1}^R \V{a}_r^{(1)} \circ \V{a}_r^{(2)} \circ \cdots \circ \V{a}_r^{(N)},
\end{equation}
where \emph{factor vector} $\V{a}_r^{(n)}$ has length $I_n$. 
Each rank-one tensor is called a \emph{component}.
The collection of all factor vectors for a given mode is called a 
\emph{factor matrix}:
\begin{displaymath}
  \Mn{A}{n} =
  \begin{bmatrix}
    \MnC{A}{n}{1} &
    \MnC{A}{n}{2} &
    \cdots &
    \MnC{A}{n}{r}
  \end{bmatrix}
  \in\R^{I_n \times R}.
\end{displaymath}
The mode-$n$ matricization of $\T{\tilde{X}}$ can be written in terms
the factor matrices as
\begin{equation}\label{eq:Zn}
  \M{\tilde{X}}_{(n)} = \M{A}^{(n)}\M{Z}^{(n)\trans}
  \qtext{where}
  \M{Z}^{(n)}=\M{A}^{(N)}\odot \cdots  \M{A}^{(n+1)}\odot
  \M{A}^{(n-1)} \odot \cdots \odot \M{A}^{(1)}.   
\end{equation}
We may alternatively represent \cref{eqn:cpform} by normalizing all
the factor vectors to unit length and expressing the product of the
normalization factors as a scalar weight $\lambda_r$ for each
component: 
\begin{equation}\label{eq:cpform+lambda}
  \T{\tilde{X}} = \sum_{r=1}^R \lambda_r \; \V{a}_r^{(1)} \circ
  \V{a}_r^{(2)} \circ \cdots \circ \V{a}_r^{(N)}.
\end{equation}

\subsection{Randomized Least Squares and Sketching}
\label{sec:sketching}
Sketching is a technique for solving linear algebra problems by constructing 
a smaller problem whose solution is a 
reasonable approximation to the original problem with high probability~\cite{sketching}. 
For instance, a large matrix may be formed by applying random sampling or random 
projections to form a smaller \emph{sketch} matrix. 
We focus on the case where a regression
problem $\min_{\V{x}}\| \M{A}\V{x} - \V{b} \|_2$ (with overdetermined
$\M{A} \in \mathbb{R}^{n\times d}$) is transformed using some random
projection $\M{M} \in \mathbb{R}^{\samplesize\times n}$, with $\samplesize\ll
n$, such that an exact solution to $\min_{\V{x}}\| \M{M}\M{A}\V{x} -
\M{M}\V{b} \|_2$ is an approximate solution to the original
problem~\cite{rokhlintygert,DrMaMuSa11,blendenpik}.

There are two leading sampling approaches for randomized least squares
problems. One involves sampling from the coefficient matrix in a
weighted manner, e.g. by computing (or estimating) \emph{leverage
  scores} for each row and sampling based on their distribution. The
other approach, which we use in this work, is to \emph{mix} the
coefficient matrix with the intention of evenly distributing leverage scores
across all rows in such a way that \emph{uniform} sampling is
effective. 
\begin{definition}
Given $\M{A} \in \mathbb{R}^{n\times d}, n > d$, the \emph{leverage score} of row $i$ of $\M{A}$ is $l_i = ||\M{U}(i,:)||_2^2$
for $i \in \{1,\dots ,n\}$ where $\M{U}$ contains the $d$ left singular vectors of $\M{A}$.
\end{definition}
Thus, the leverage score of a row corresponds in some sense to the
importance of that row in constructing the column-space of the
coefficient matrix. 

In 2007, Drineas et al.~\cite{DrMaMuSa11} presented a relative-error
least squares algorithm that gives a $1 + \epsilon$
approximation. They first mix the coefficient matrix using a
randomized Hadamard transform (discussed later), and then sample 
\begin{equation}
\label{eqn:numsamples}
\bigO{\max\{ d\log{(n)}\log{(d\log{(n)})}, d\log{(nd)}/\epsilon\}}
\end{equation}
rows of the resulting matrix before computing the solution using
normal equations. The dependence of the sampling size on $\epsilon$
makes this algorithm fairly impractical for typical direct
solvers. However, subsequent work by Rokhlin and
Tygert~\cite{rokhlintygert} applied a related sketching strategy to
the \emph{preconditioning} of a Krylov-subspace method, establishing a
relationship between sample size and condition number.  

Avron et al.\@ synthesized these concepts into a high-performance
solver called Blendenpik~\cite{blendenpik}. They first apply a
randomized Hadamard transform (or similar transform),
compute a QR decomposition of the result, and use its
$R$-factor as a preconditioner
for the standard LSQR solver. Additionally they show that the
condition number of their system depends on the \emph{maximal}
leverage score of the matrix, referred to as \emph{coherence}. 
\begin{definition}[\cite{coherence,candesrecht}]
\emph{Coherence} is  the maximum leverage score of $\M{A}$, i.e.,
\[ \mu(\M{A}) = \max_{i \in \{1,\dots,n\}} l_i, \]
where $l_i$ is the leverage score of the $i$th row of $\M{A}$. It
holds that $\frac{d}{n} \leq \mu(\M{A}) \leq 1$. 
\end{definition}
Intuitively, if a row of a matrix $\M{A}$ contains the only nonzero in a column then $\mu(\M{A})=1$ and any row-sampling $\M{SA}$ must include that row (which has leverage score 1) or it will be rank deficient. If coherence is close to 1, a uniform row-sampling is likely to be \emph{nearly} rank-deficient, leading to a poorly conditioned reduced-size least squares problem and an inaccurate approximate solution vector. 

In~\cref{sec:cprand} we prove that the standard formulation of CP-ALS
may increase incoherence, making uniform sampling effective in many
situations. However, in order to guarantee incoherence (w.h.p.)
regardless of input, it is necessary to preprocess with a mixing step.  

This mixing strategy relates to a more general class of
transformations that rely on quality guarantees provided by the
Johnson-Lindenstrauss Lemma~\cite{jlt}. This lemma specifies a class
of random projections that preserve the distances between all pairs of
vectors with reasonable accuracy. The \emph{fast}
Johnson-Lindenstrauss transform (FJLT) is able to avoid explicit
matrix multiplications by utilizing efficient algorithms such as the
fast Fourier transform (FFT), discrete cosine transform (DCT), or
Walsh-Hadamard Transform (WHT)~\cite{fjlt}. These transforms can
operate on a vector $\V{x} \in \mathbb{R}^n$ in $\inlineBigO{n\log_2 n}$
time. What these algorithms have in common is that
they improve incoherence, mixing information across every element of a vector, while at the
same time being orthogonal operations (i.e., a change of basis).  
The theoretical quality guarantees and theoretical computational costs are the same for all fast transforms.

The FJLT consists of three steps. First, each row of the coefficient
matrix is sign-flipped with probability $1/2$. This is equivalent to
computing $\M{D}\M{A}$ with diagonal matrix $\M{D}\in
\mathbb{R}^{n\times n}$, where each diagonal element is $\pm 1$ with
equal probability. Second, we apply the fast mixing operation
$\mathcal{F}$. Third we uniformly sample $\samplesize$ rows of the
result with uniform probability. Thus, the entire operation can be
written out as $\M{S}\mathcal{F}\M{D}\M{A}$, where $\M{S}$ is a row-sampling operator (containing unit row vectors $\V{e}_{i}$, for each sampled row $i$).
The reasoning behind first
applying $\M{D}$ is that input data is often sparse in the frequency
domain, and randomly flipping the signs of the coefficient matrix is an
orthogonal operation that spreads out the frequency domain of
the signal~\cite{fjlt}. In this paper we will exclusively use the FFT
for the $\mathcal{F}$ operation due to its ease of reproducibility and its efficiency in
MATLAB. 
Though portable, this has the result of making all data complex-valued, which we discuss in more detail later on.
We observed that using alternative transforms had no effect on the quality of our solutions, but real-valued transforms were slower to apply because of a lack of efficient implementations within MATLAB.

\section{Algorithms} \label{sec:algorithm} 
In this section we introduce CPRAND, which samples without mixing, and
CPRAND-MIX, which efficiently applies mixing before sampling. We first
recall the standard CP-ALS method, which is the starting point for our
modifications. 

\subsection{CP-ALS}

The standard method for fitting the CP model is alternating least
squares (CP-ALS) \cite{PARAFAC,Kolda:2009}. The method alternates
among the modes, fixing every factor matrix but $\Mn{A}{n}$ and
solving for it. From \cref{eq:Zn}, 
we see that we can find $\Mn{A}{n}$ by solving the linear least squares
problem given by
\begin{equation}
\label{eq:lls}
\argmin_{\M{A}^{(n)}} \|\M{X}_{(n)} - \M{A}^{(n)}\M{Z}^{(n)\trans}\|_F.
\end{equation}
In CP-ALS, we work with the normal equations for \cref{eq:lls}:
\begin{displaymath}
\M{X}_{(n)}\M{Z}^{(n)} = \M{A}^{(n)}(\M{Z}^{(n)\trans}\M{Z}^{(n)}),
\end{displaymath}
and solve for $\M{A}^{(n)}$ for given $\M{X}_{(n)}$ and $\M{Z}^{(n)}$.
By identity \cref{eq:KRGram}, we have 
\begin{displaymath}
\M{Z}^{(n)\trans}\M{Z}^{(n)} = \M{A}^{(N)\trans}\M{A}^{(N)} \circledast \dots \circledast \M{A}^{(n+1)\trans}\M{A}^{(n+1)} \circledast \M{A}^{(n-1)\trans}\M{A}^{(n-1)} \circledast \cdots \circledast \M{A}^{(1)\trans}\M{A}^{(1)}.
\end{displaymath}

The CP-ALS algorithm~\cite{Kolda:2009} is presented
in~\cref{alg:cpals}. Note the step where vector $\bm{\lambda}$ stores
normalization values of each column so that the final approximation is
as in \cref{eq:cpform+lambda}; this normalization helps alleviate
issues due to scaling ambiguity.

The initialization of the factor matrices in \LineRef{line:cpals:init}
can impact the performance of the algorithm.
There are many possible ways to do the the initialization.
One way is to
initialize is to set $\Mn{A}{n}$ to be the leading $R$ left singular
vectors of the mode-$n$ unfolding, $\M{X}_{(n)}$, and we call this
HOSVD initialization, as it corresponds to the factor matrices
in the rank-$(R{\times} {\cdots} {\times} R)$ Higher-Order SVD. 
A less expensive but less effective initialization is to
choose random factor matrices.

\begin{algorithm}
  \caption{CP-ALS}
  \label{alg:cpals}
  \begin{algorithmic}[1]\footnotesize
    \Function{$[\bm{\lambda},\set{\M{A}^{(n)}}]=$ CP-ALS}{$\T{X},R$}\Comment{$\T{X}\in\mathbb{R}^{I_1\times \cdots \times I_N}$}
    \State \label{line:cpals:init} Initialize factor matrices $\M{A}^{(2)}, \dots, \M{A}^{(N)}$
    \Repeat
    \For{$n=1,\dots, N$}
      \State $\M{V} \gets \M{A}^{(N)\trans}\M{A}^{(N)} \circledast \dots \circledast \M{A}^{(n+1)\trans}\M{A}^{(n+1)} \circledast \M{A}^{(n-1)\trans}\M{A}^{(n-1)} \circledast \cdots \circledast \M{A}^{(1)\trans}\M{A}^{(1)}$\label{line:cpals:Gram}
      \State \label{line:cpals:KR} $\M{Z}^{(n)} \gets \M{A}^{(N)}\odot \cdots  \odot \M{A}^{(n+1)}\odot \M{A}^{(n-1)} \odot \cdots \odot \M{A}^{(1)}$
      \State \label{line:cpals:MTTKRP} $\M{W} \gets \M{X}_{(n)}\M{Z}^{(n)}$
      \State \label{line:cpals:solve} Solve $\M{A}^{(n)}\M{V} = \M{W}$ for $\M{A}^{(n)}$        
      \State Normalize columns of $\M{A}^{(n)}$ and update $\bm{\lambda}$
    \EndFor
    \Until termination criteria met
    \State \textbf{return} $\bm{\lambda}$, factor matrices $\set{\M{A}^{(n)}}$
    \EndFunction
  \end{algorithmic}
\end{algorithm}

\subsubsection{Cost}
We consider the cost of a single outer iteration of CP-ALS.
In \LineRef{line:cpals:Gram}, the cost of computing the $m$th Gram
matrix is $R^2 I_m$; so the entire cost is  
$R^2 \sum_{m\neq n} I_m$ flops to compute all the Gram matrices plus
$\inlineBigO{R^2N}$ to multiply them all together to form $\M{V}$. 
The combination of \LineRef{line:cpals:KR} and
\LineRef{line:cpals:MTTKRP} form an operation  called the
matricized tensor times Khatri-Rao 
  product (MTTKRP). This is a frequent target of
optimization~\cite{stop1,phan,SPLATT,dfacto}.
The Khatri-Rao product in \LineRef{line:cpals:KR}
requires $\inlineBigO{R \prod_{m\neq n} I_m}$ flops (flops may be reduced at the cost of 
more memory by storing and reusing partial products, but we ignore this detail in our discussion).   
The computation of $\M{X}_{(n)}\M{Z}^{(n)}$ in
\LineRef{line:cpals:MTTKRP} is the most expensive step, with a cost
of $2R\prod_m I_m$ flops.  
We note that Phan et al.~\cite{phan} give a clever reorganization of
\LineRef{line:cpals:KR} and \LineRef{line:cpals:MTTKRP}, which avoids data
movement and so may reduce overall runtime. 
The cost of solving the linear system in \LineRef{line:cpals:solve}
using Cholesky decomposition is dominated by
the triangular solves with the Cholesky factors, which requires
$2R^2I_n$ flops. 
The overall cost of each outer iteration (updating each factor matrix
once) is $\inlineBigO{NR\prod_{n}I_n}$.  

If HOSVD initialization is used in \LineRef{line:cpals:init}, then the
costs of forming the mode-$n$ Gram matrix is $I_n \prod_{m}I_m$ and
the cost of computing the eigenvectors is $I_n^3$. Hence, the total
initialization cost is $(\sum_{n>1} I_n) \prod_m I_m + \sum_{n>1}
I_n^3$.
 
\subsection{CPRAND} \label{sec:cprand}
Consider the overdetermined least squares problem in \cref{eq:lls},
which is convenient to rewrite as
\begin{equation}\label{eq:lsalt} 
\argmin_{\M{A}^{(n)}}  \| \Mn{Z}{n} \M{A}^{(n)\trans} - \M{X}_{(n)}^\trans \|_F.
\end{equation}
The simplest sampling algorithm uniformly samples rows from
$\M{Z}_{(n)}$ and the corresponding rows from $\M{X}_{(n)}^{\trans}$.
We let $\samplesize$ denote the number of desired number of samples.
Without loss of generality, we assume $S > \max \set{I_1, \dots, I_N, R}$.
Let $\mathcal{S}$ denote the samples from $\set{1,\dots,\prod_{m\neq
    n}I_m}$ such that $| \mathcal{S} | = S$.
We do \emph{uniform sampling with replacement} which means that every
row has equal chance of being selected and the same row
may be selected more than once.
We can express the sampling operation as the application of a selection matrix $\M{S} \in \mathbb{R}^{\samplesize \times
  \prod_{\modeidx\neq n} I_\modeidx}$, where the rows of $\M{S}$ are
rows of the $\prod_{\modeidx\neq n} I_\modeidx \times \prod_{\modeidx\neq n}
I_\modeidx$ identity matrix. 

Forming the (full) Khatri-Rao product $\Mn{Z}{n}$ is expensive, so we
want to compute the sampled matrix $\M{S}\Mn{Z}{n}$ without ever
explicitly forming $\Mn{Z}{n}$.
Consider sampling the $j$th row of
$\M{Z}_{(n)}$. 
Using the mapping in \cref{eqn:unfoldmapping}, we can map $j$ to
indices $(i_1,\dots,i_{n-1},i_{n+1},\dots,i_N)$ (the $n$th index
is omitted). In fact, the $j$th row of $\Mn{Z}{n}$ is the Hadamard
product of the appropriate rows of the factor matrices, i.e.,
\begin{displaymath}
  \Mn{Z}{n}(j,:) = 
  \Mn{A}{1}(i_1,:) \circledast \cdots \circledast  
  \Mn{A}{n-1}(i_{n-1},:) \circledast
  \Mn{A}{n+1}(i_{n+1},:) \circledast \cdots \circledast 
  \Mn{A}{N}(i_{N},:) .
\end{displaymath}
This is illustrated in \cref{fig:skr} for a three-way tensor.
We give the algorithm for computing \emph{sampled Khatri-Rao} (SKR) in~\cref{alg:skr}, where
\texttt{idxs} is the set of tuples
\begin{displaymath}
  \{i_1^{(j)},\dots,i_{n-1}^{(j)},i_{n+1}^{(j)},\dots,i_N^{(j)}\} 
  \qtext{for} j \in \mathcal{S}.
\end{displaymath}
We assume these tuples are stacked in matrix
form for efficiency. Thus, each multiplicand $\M{A}_S^{(m)}$ is of size
$\samplesize \times R$. 

\begin{algorithm}[tbhp]
  \caption{Sampled Khatri-Rao Product}
  \label{alg:skr}
  \begin{algorithmic}[1]\footnotesize
    \Function{$\M{Z}_S=$ SKR}{$\M{S},\M{A}^{(N)},\dots,\M{A}^{(n+1)},\M{A}^{(n-1)},\dots,\M{A}^{(1)}$} 
    \State Retrieve $\mathtt{idxs}$ from $\M{S}$
    \State $\M{Z}_S \gets \M{1}$ \Comment{$\M{1} \in \mathbb{R}^{\samplesize \times R}$}
    \For{$m=1,\dots,n-1,n+1,\dots,N$}
      \State $\M{A}_S^{(m)} \gets \M{A}^{(m)}(\mathtt{idxs}(:,m),:)$ \Comment{MATLAB-style indexing}
      \State $\M{Z}_S \gets \M{Z}_S \circledast \M{A}_S^{(m)}$
    \EndFor
    \State \Return $\M{Z}_S$
    \EndFunction
  \end{algorithmic}
\end{algorithm}
\begin{figure}[t]
\centering 
\includegraphics[width=0.7\linewidth]{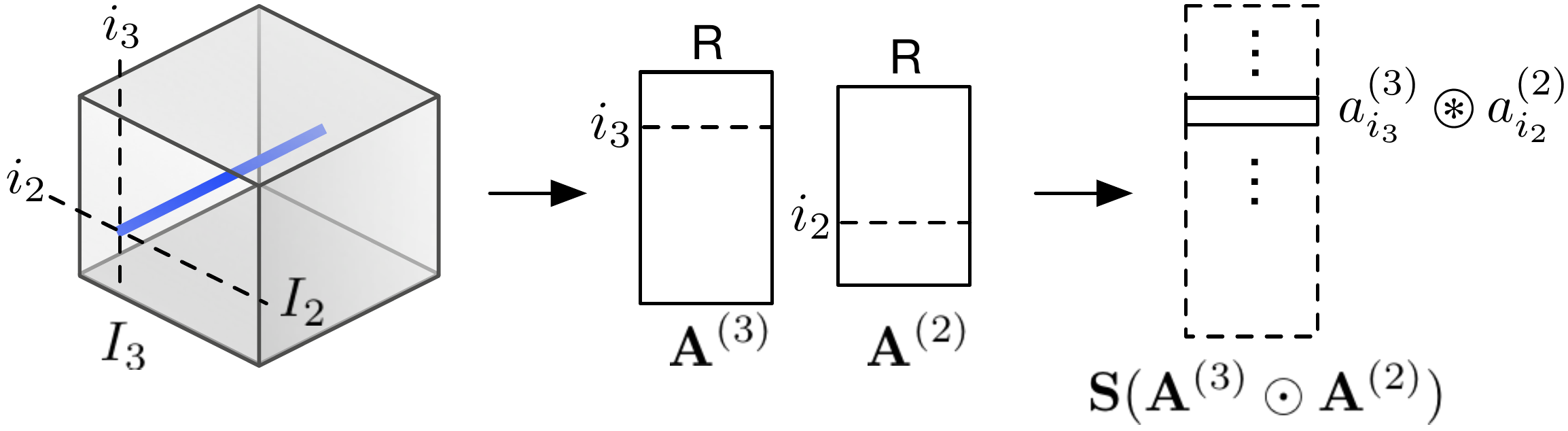}
\caption{Sampled Khatri-Rao rows correspond to sampled fibers in
  $\T{X}$.}
\label{fig:skr}
\end{figure}

In the same way we wanted to avoid forming $\Mn{Z}{n}$ explicitly, we
also want to avoid forming 
$\M{X}_{(n)}$ (i.e., in this case that means avoiding the data
movement). 
Instead, we observe that if we
sample the $j$th row of $\M{X}_{(n)}$, 
then we want the fiber
$\V{x}_{i_1,\dots,i_{n-1},:,i_{n+1},\dots,i_N}$ where we are using the same mapping
as for SKR. 
Thus, we can avoid matricization
and pull entries from the tensor directly to form
$\M{S}\M{X}_{(n)}^{\trans}$.

Our randomized version of CP-ALS is named CPRAND
and shown in \cref{alg:rcpals}, where we solve a sampled version of
the least squares problem in \cref{eq:lsalt}.
In \LineRef{line:rcpals:SKR}, we use the SKR from \cref{alg:skr} to
get the sampled version of $\Mn{Z}{n}$. 
In \LineRef{line:rcpals:Xs}, we sample rows of $\M{X}_{(n)}^{\trans}$.
In \LineRef{line:rcpals:QR}, we solve the sampled least squares
problem where 
the coefficient matrix $\M{Z}_S$ is of 
size $\samplesize \times R$ and
the corresponding sampled right hand side $\M{X}_S^{\trans}$ is of size
$\samplesize \times I_n$.
The solution $\Mn{A}{n}$ is of size $I_n \times R$.

In all of the experiments in this paper, a sample size of $\samplesize=10R\log{(R)}$ has proven sufficient, provided that the data is incoherent.

\begin{algorithm}[tbhp]
  \caption{CPRAND}
  \label{alg:rcpals}
  \begin{algorithmic}[1]\footnotesize
    \Function{
      CPRAND}{$\T{X},R,\samplesize$}\Comment{$\T{X}\in\mathbb{R}^{I_1\times \cdots \times I_N}$}
    \State Initialize factor matrices $\M{A}^{(2)}, \dots, \M{A}^{(N)}$
    \Repeat 
    \For{$n=1,\dots, N$}
      \State Define sampling operator $\M{S} \in \mathbb{R}^{\samplesize \times \prod_{\modeidx\neq n} I_\modeidx}$
      \vspace{0.5mm}
      \State \label{line:rcpals:SKR} $\M{Z}_S \gets
      \text{SKR}(\M{S},\M{A}^{(1)},\dots,\M{A}^{(n-1)},\M{A}^{(n+1)},\dots,\M{A}^{(N)})$
      \vspace{1mm}
      \State \label{line:rcpals:Xs} $\M{X}_S^{\trans} \gets \M{S}\M{X}_{(n)}^\trans$
      \State \label{line:rcpals:QR} $\M{A}^{(n)} \gets \displaystyle
      \argmin_{\M{A}} \left\|\M{Z}_S \M{A}^{\trans} - \M{X}_S^{\trans} \right\|_F$
      \State Normalize columns of $\M{A}^{(n)}$ and update $\bm{\lambda}$. 
    \EndFor
    \Until termination criteria met
    \State \textbf{return} $\bm{\lambda}$, factor matrices $\set{\M{A}^{(n)}}$
    \EndFunction
  \end{algorithmic}
\end{algorithm}

\subsubsection{Cost}
\label{sec:rcpals-cost}
The arithmetic cost of~\cref{alg:rcpals} comprises the cost of
sampling, to set up the smaller least squares problem, and the cost of
solving the least squares problem.
Generating $S$ random multiindices is $O(\samplesize N)$ operations.
Sampling the Khatri-Rao product in \LineRef{line:rcpals:SKR} using~\cref{alg:skr} requires $\samplesize R(N{-}1)$ flops, as each of the $\samplesize$ length-$R$ rows of $\M{Z}_S^{(n)}$ is formed as a Hadamard product of $N{-}1$ rows of the fixed factor matrices.
Sampling $\samplesize$ fibers from $\T{X}$ to form $\M{X}_S^{(n)}$
requires no flops, but it does require irregular data access to construct the $I_n\times \samplesize$ matrix and is actually the most time-consuming operation for large tensors.
Using QR to solve the reduced least squares problem in
\LineRef{line:rcpals:QR}
is $2 \samplesize R^2$ flops, the cost of computing
$\M{X}_S^{(n)}\M{Q}$ (applying $\M{Q}$) is $2\samplesize RI_n$
operations, and the cost of the triangular solve is $R^2I_n$
operations.
Assuming $I_n > R$ and $\samplesize > R$, the leading order cost is $2\samplesize R I_n$ operations.
The overall cost of each outer iteration (updating each mode once) is $\inlineBigO{\samplesize R\sum_{n}I_n}$. 

\subsubsection{Coherence in CPRAND}
The effectiveness of CPRAND depends on the coherence of coefficient matrix $\M{Z}^{(n)}$. Since $\M{Z}^{(n)}$ is formed as the Khatri-Rao product of factor matrices, it is natural to ask what effect the Khatri-Rao product has on coherence. 

More rigorously we now show that there is, at the very least, a
multiplicative reduction in coherence when we form $\Mn{Z}{n}$. We
begin by bounding the coherence of the Kronecker product and use this
to bound the coherence of the Khatri-Rao product.
\begin{lemma}
Given $\M{A} \in \mathbb{R}^{I \times J}$ and $\M{B} \in \mathbb{R}^{K \times L}$, $\mu(\M{A} \otimes \M{B}) = \mu(\M{A})\mu(\M{B})$.
\end{lemma}
\begin{proof}
We take the reduced QR factorizations of the two terms and then apply~\cref{eqn:krondist}:
\[\M{A}\otimes \M{B} = \M{Q}_A\M{R}_A \otimes \M{Q}_B\M{R}_B = (\M{Q}_A \otimes \M{Q}_B)(\M{R}_A \otimes \M{R}_B)\]
This is a reduced QR factorization of $\M{A}\otimes \M{B}$, and the $\M{Q}$ factor has rows 
${\M{Q}_A(i,:)} \otimes {\M{Q}_B(j,:)}$ for every possible pair $(i,j)$. We know through simple arithmetic that $\|\V{a} \otimes \V{b}\| = \|\V{a}\| \|\V{b}\|$, so the max row norm of $\M{Q}_A \otimes \M{Q}_B$ is the product of the max row norms of $\M{Q}_A$ and $\M{Q}_B$.
\end{proof}
\begin{lemma} \label{lem:krp}
Given $\M{A} \in \mathbb{R}^{I \times J}$ and $\M{B} \in \mathbb{R}^{K \times L}$, $\mu(\M{A} \odot \M{B}) \leq \mu(\M{A})\mu(\M{B})$.
\end{lemma}
\begin{proof}
We again take the reduced QR factorizations of the two terms and then apply~\cref{eqn:krdist}: 
\begin{equation*}
\M{A} \odot \M{B} = \M{Q}_A\M{R}_A \odot \M{Q}_B\M{R}_B = (\M{Q}_A \otimes \M{Q}_B)(\M{R}_A \odot \M{R}_B).
\end{equation*}
We then take the QR decomposition of the second term:
\begin{equation*}
(\M{Q}_A \otimes \M{Q}_B)(\M{R}_A \odot \M{R}_B) = (\M{Q}_A \otimes \M{Q}_B)\M{Q}_R\M{R}_R = \M{\hat{Q}}\M{R}_R.
\end{equation*}
The reduced $\M{Q}$ term for $\M{A} \odot \M{B}$ is $\M{\hat{Q}} = (\M{Q}_A \otimes \M{Q}_B)\M{Q}_R$, so the norm of row $i$ in $\M{\hat{Q}}$ is $\hat{l}_i$, the $i$th leverage score of $\M{A} \odot \M{B}$. 
Letting $\V{\hat{q}}^\trans$ be row $i$ of $\M{Q}_A \otimes \M{Q}_B$,
\[ \hat{l}_i = \|\V{\hat{q}}_i^\trans \M{Q}_R \| = \|\M{Q}^\trans_R\V{\hat{q}}_i\| \leq \|\M{Q}_R^\trans\|\xspace\|\V{\hat{q}}_i\|. \]
$\M{Q}_R^\trans$ has orthonormal rows, so $\|\M{Q}_R^\trans\|=1$, yielding:
\begin{gather*}\mu(\M{A}\odot \M{B}) = \mu(\hat{\M{Q}}) = \max_i \hat{l}_i \quad
\leq \quad \max_i\|\V{\hat{q}}_i\| = \mu(\M{Q}_A \otimes \M{Q}_B) = \mu(\M{A})\mu(\M{B}).
\end{gather*}
\end{proof}
Using the notation of the prior proof, the exact coherence expands to:
\[ \mu(\M{A} \odot \M{B}) = \max_i \sqrt{\V{\hat{q}}_i^\trans\M{Q}_R\M{Q}_R^\trans\V{\hat{q}}_i}, \]
where the amount of truncation in $\M{Q}_R$ corresponds to how loose the inequality is. 
This bound is tight, e.g., for ${\M{A} = (1, 1)^\trans}$ and ${\M{B} = (1, -1)^\trans}$, but we typically see a large factor of coherence reduction.

\subsection{CPRAND-MIX}
\label{sec:cprand-mix}

\cref{lem:krp} shows that the Khatri-Rao product inherits incoherence from its factors. 
However, if the individual factor matrices happen to be highly
coherent, CPRAND may fail to converge.
We can
prevent this by \emph{mixing} the  terms before sampling occurs, as in~\cref{alg:rcpalsfft}. 
Consider the least squares problem in CP-ALS in \cref{eq:lsalt}.
Recall the FJLT introduced in~\cref{sec:sketching}. We \emph{could} apply such a transformation directly to the inner iteration before sampling, yielding
\begin{displaymath}
\argmin_{\M{A}^{(n)}} \| \mathcal{F} \M{D} \M{Z}^{(n)}\M{A}^{(n)\trans}  - \mathcal{F} \M{D}\M{X}_{(n)}\|_F,
\end{displaymath}
where $\M{D}$ is a diagonal random sign matrix and $\mathcal{F}$ is the FFT matrix. 
However, this would involve mixing the matricization of $\T{X}$ in each mode, as well as forming and mixing the full Khatri-Rao product at each iteration. 
Instead of mixing the rows of the entire Khatri-Rao product $\M{Z}^{(n)}$, we mix the rows of each factor $\M{A}^{(m)}$ individually. 
Using the distributive property in equation~\cref{eqn:krdist}, we see that this is equivalent to applying a Kronecker product of mixing terms:
\[\M{\hat{Z}}^{(n)} = \bigodot_{\substack{\modeidx=N \\ \modeidx\neq n}}^1\mathcal{F}_\modeidx\M{D}_\modeidx\M{A}^{(\modeidx)}
= \left(\bigotimes_{\substack{\modeidx=N\\\modeidx\neq n}}^1 \mathcal{F}_\modeidx\M{D}_\modeidx\right) \M{Z}^{(n)}
=\left(\bigotimes_{\substack{\modeidx=N\\\modeidx\neq n}}^1 \mathcal{F}_\modeidx\right)\left(\bigotimes_{\substack{\modeidx=N\\\modeidx\neq n}}^1 \M{D}_\modeidx\right) \M{Z}^{(n)}.\]
Applying this operation to the CP least squares problem leads to the mixed formulation of the least squares problem:
\begin{equation}
\label{eq:cpalsmixed}
\argmin_{\M{A}^{(n)}} \left\| 
  \left(\bigotimes_{\substack{\modeidx=N\\\modeidx\neq n}}^1
    \mathcal{F}_\modeidx\M{D}_\modeidx \right)
  \M{Z}^{(n)}\M{A}^{(n)\trans}
  -      
  \left(\bigotimes_{\substack{\modeidx=N\\\modeidx\neq n}}^1
    \mathcal{F}_\modeidx\M{D}_\modeidx\right)
  \M{X}_{(n)}^{\trans} 
\right\|_F.
\end{equation}
The Kronecker product preserves orthogonality (unitarity), so \cref{eq:cpalsmixed} is equivalent to \cref{eq:lsalt}.

Note that while we do not prove that applying uniform sampling to the
columns of \cref{eq:cpalsmixed} is an FJLT (we conjecture it is), we
know that $\mu(\M{\hat{Z}}^{(n)})$ is upper bounded
by~\cref{lem:krp}, so we expect uniform sampling to be sufficient for
an accurate approximate solution. 
Using equation~\cref{eqn:ttensor} we can see that the second term of \cref{eq:cpalsmixed} is equivalent to
\begin{equation}
\label{eq:unmixX}
\bigl( \M{D}_n\mathcal{F}_n^{*}\M{\hat{X}}_{(n)} \bigr)^{\trans}, 
\quad\text{where}\quad \T{\hat{X}} = \T{X} \times_1 \mathcal{F}_1\M{D}_1 \cdots\times_N  \mathcal{F}_N\M{D}_N,
\end{equation}
where we note that $\M{D}_n$ is its own inverse, $\mathcal{F}_n$ is
unitary in the case of the FFT, and the asterisk denotes the conjugate transpose.
Using a uniform sampling matrix $\M{S}$, our reduced problem has the form
\begin{equation}
  \label{eq:cpalsmixedsampled}
  \argmin_{\M{A}^{(n)}} 
  \left\| 
    \left(\M{S}\M{\hat Z}^{(n)}\right)\M{A}^{(n)\trans} -
    \M{D}_n\mathcal{F}_n^{*}
    \left(
      \M{S} 
      \M{\hat{X}}_{(n)}^{\trans}
    \right)^{\trans}
  \right\|_F.
\end{equation}

This approach is presented as CPRAND-MIX in \cref{alg:rcpalsfft2}.
We highlight two computational optimizations in \cref{alg:rcpalsfft2}.
First, by \cref{eq:unmixX}, we can apply a single upfront mixing of
the tensor in all modes in a preprocessing step
(\LineRef{line:rcpalsfft2:mixX}). Then, at each inner iteration for
mode $n$, we \emph{unmix} the tensor in only mode $n$.
Furthermore, this unmixing can be done \emph{after} the columns of $\M{\hat{X}}_{(n)}$ are sampled, as shown in \LineRef{line:rcpalsfft2:Xs}.
Second, we can avoid mixing the entire Khatri-Rao product matrix at each step because only one factor matrix changes each iteration.
Thus, we mix the $n$th factor matrix to produce $\M{\hat{A}}^{(n)}$ in
\LineRef{line:rcpalsfft2:mixA} only once per outer iteration,
immediately after computing $\M{A}^{(n)}$. Then, we can sample the mixed Khatri-Rao product without forming it explicitly using \cref{alg:skr}.

\begin{algorithm}
  \caption{CPRAND-MIX}
  \label{alg:rcpalsfft2}
  \begin{algorithmic}[1]\footnotesize
    \Procedure{CPRAND-MIX}{$\T{X},R,S$}\Comment{$\T{X}\in\mathbb{R}^{I_1\times \cdots \times I_N}$}
    \State Initialize factor matrices $\M{A}^{(m)}, \; m \in \{ 2 \dots N\}$
    \State Define random sign-flip operators $\M{D}_m$ and unitary matrices $\M{\mathcal{F}}_m, \; m \in \set{1,\dots, N}$
    \State Mix factor matrices: $\M{\hat{A}}^{(m)}\gets \M{\mathcal{F}}_m \M{D}_m \M{A}^{(m)}, \; m \in \{ 2 \dots N\}$
    \State \label{line:rcpalsfft2:mixX} Mix tensor: $\T{\hat{X}} \gets \T{X} \times_1 \mathcal{F}_1\M{D}_1 \cdots\times_N  \mathcal{F}_N\M{D}_N$
    \Repeat
    \For{$n=1,\dots, N$}
    \State Define sampling operator $\M{S} \in \mathbb{R}^{\samplesize \times \prod_{\modeidx\neq n} I_\modeidx}$
      \vspace{1mm}
    \State $\M{\hat{Z}}_S \gets  SKR(\M{S},\M{\hat{A}}^{(N)},\dots,\M{\hat{A}}^{(n+1)},\M{\hat{A}}^{(n-1)},\dots,\M{\hat{A}}^{(1)})$
    \State \label{line:rcpalsfft2:Xs} $\M{\hat{X}}_S^{\trans} \gets \M{D}_n\mathcal{F}_n^{*} \left(\M{S}\M{\hat{X}}_{(n)}^\trans\right)^{\trans}$
    \State \label{line:rcpalsfft2:LS} $\M{A}^{(n)} \gets
    \displaystyle \argmin_{\M{A}} 
    \left\| \M{\hat{Z}}_S \M{A}^{\trans}  
      - \M{\hat{X}}_S^{\trans} \right\|_F$
    subject to $\M{A}$ being real-valued
    \State Normalize columns of $\M{A}^{(n)}$ and update $\bm{\lambda}$
      \vspace{1mm}
    \State \label{line:rcpalsfft2:mixA} $\M{\hat{A}}^{(n)} \gets \M{\mathcal{F}}_n \M{D}_n \M{A}^{(n)}$
    \EndFor
    \Until termination criteria met
    \State \textbf{return} $\bm{\lambda}$, factor matrices $\set{\M{A}^{(n)}}$
    \EndProcedure
  \end{algorithmic}
\end{algorithm}
 
Next, we point out a subtlety within \cref{alg:rcpalsfft2} due to our use of the FFT, which is complex valued.
Assuming the input tensor is real valued, we seek a CP decomposition that is also real valued.
However, the least squares problem in \LineRef{line:rcpalsfft2:LS} involves complex-valued matrices, and the solution can be also be complex valued.
We note that if the least squares problem is mixed but not sampled, and if the Khatri-Rao product is full rank, then the solution would still be real valued.
However, because sampling implies that we are solving the original least squares problem approximately, the solution can drift into the complex plane.
In order to maintain a real-valued CP approximation, we solve the
least squares problem over only real values by using the equivalence
\begin{equation}
\label{eq:rcLS}
\argmin_{\V{x}\in\mathbb{R}^n} \|\M{A}\V{x}-\V{b}\|_2  
= \argmin_{\V{x}\in\mathbb{R}^n} \left\|\begin{bmatrix} \Re(\M{A}) \\
    \Im(\M{A}) \end{bmatrix} \V{x}-
  \begin{bmatrix} \Re(\V{b}) \\ \Im(\V{b}) \end{bmatrix} \right\|_2,
  \end{equation}
where $\M A\in\mathbb{C}^{m\times n}$ and
$\V{b}\in\mathbb{C}^m$.
When a real-valued orthogonal matrix is used instead of the FFT (such as the DCT or the WHT), the aforementioned subtlety can be ignored.
In fact, the algorithm can be further simplified.
We note that if we write the new ALS update in terms of the pseudoinverse we get the following:
\begin{equation*}
\label{eqn:mixupdate1}
\M{A}^{(n)} \gets \M{D}_n\mathcal{F}_n^{*}(\M{\hat{X}}_{(n)}\M{S}^T)[(\M{S}\M{\hat{Z}}^{(n)})^\trans]^\dag.
\end{equation*}
This update implies that we store the unmixed factor matrices, mix them before each iteration, and then unmix afterwards. 
We can actually avoid this process by maintaining the factor matrices in their mixed state for the duration of the algorithm:
\begin{equation*}
\label{eqn:mixupdate2}
\M{\hat{A}}^{(n)} \gets \M{\mathcal{F}}_n \M{D}_n \M{A}^{(n)} = \M{\hat{X}}_{(n)}\M{S}^T[(\M{S}\M{\hat{Z}}^{(n)})^\trans]^\dag.
\end{equation*}
This transformation yields the following least squares problem:
\begin{equation*}
\argmin_{\M{\hat A}^{(n)}} \left\| \M{\hat A}^{(n)}(\M{S}\M{\hat Z}^{(n)})^\trans - \M{\hat{X}}_{(n)} \M{S}^\trans\right\|_F,
\end{equation*}
which is equivalent to \LineRef{line:rcpals:QR} in \cref{alg:rcpals}, just in the mixed basis.
Thus, in the case of real-valued transforms, it is sufficient to mix $\T{X}$, call CPRAND as a subroutine, and then unmix the solution factors, as shown in~\cref{alg:rcpalsfft}. 

\begin{algorithm}
  \caption{CPRAND-PREMIX}
  \label{alg:rcpalsfft}
  \begin{algorithmic}[1]\footnotesize
\Function{CPRAND-PREMIX}{$\T{X},R,\samplesize$}\Comment{$\T{X}\in\mathbb{R}^{I_1\times \cdots \times I_N}$}
    \State Define random sign-flip operators $\M{D}_m$ and orthogonal matrices $\M{\mathcal{F}}_m, \; m \in \set{1,\dots, N}$
    \State \label{line:rcpalsfft:mix} Mix: $\T{\hat{X}} \gets \T{X} \times_1 \mathcal{F}_1\M{D}_1 \times\cdots\times_N  \mathcal{F}_N\M{D}_N$
    \State $[\bm{\lambda},\{\M{\hat{A}}^{(n)}\}]=\text{ CPRAND}(\T{\hat{X}},R,\samplesize)$
    \For{$n=1,\dots,N$}
      \State \label{line:rcpalsfft:unmix} Unmix: $\M{A}^{(n)} = \M{D}_n \mathcal{F}_n^\trans \M{\hat{A}}^{(n)}$
    \EndFor
    \State \textbf{return} $\lambda$, factor matrices $\set{\M{A}^{(n)}}$
    \EndFunction
  \end{algorithmic}
\end{algorithm}

\subsubsection{Cost}
\label{sec:rcpalsfft-cost}

In the case of real-valued orthogonal transformations, as we show in \cref{alg:rcpalsfft}, we can implement CPRAND-MIX using CPRAND along with pre-processing (mixing) and post-processing (unmixing) steps.
The initial mixing of the tensor in line \ref{line:rcpalsfft:mix} requires significant upfront cost, but the unmixing of the factor matrices in line \ref{line:rcpalsfft:unmix} is relatively cheap.
Compared to \cref{alg:rcpals}, the dominant extra cost (of line \ref{line:rcpalsfft:mix}) is 
\begin{equation}
\label{eq:mixcost}
\bigO{\sum_{k=1}^N\prod_{\modeidx}[I_\modeidx \log I_k]} = \bigO{\left(\prod_\modeidx I_\modeidx\right) \log \left(\prod_\modeidx I_\modeidx\right)}.
\end{equation}

The cost of \cref{alg:rcpalsfft2} includes the mixing cost given by \cref{eq:mixcost}, and the cost per iteration is slightly larger than \cref{alg:rcpals}.
In particular, due to the complex values and complex arithmetic, the cost of each operation is increased by a constant factor between 2 and 4.
The leading order cost of \cref{alg:rcpals} comes from solving the least squares problem in \LineRef{line:rcpals:QR}, and the leading order cost of \cref{alg:rcpalsfft2}  also comes from solving the least squares problem (\LineRef{line:rcpalsfft2:LS}).
Following \cref{eq:rcLS}, the least squares problem is solved in real arithmetic, but the number of rows of the coefficient matrix is twice as many as in \cref{alg:rcpals}.
This yields an increase in the leading order per-iteration cost 
of a factor of 2 compared to CPRAND. 

\subsection{Stopping Criteria} \label{sec:stopping}
Given the original tensor $\T{X}$ and CP approximation $\T{\tilde{X}}$ from ~\cref{eqn:cpform}, 
the relative residual norm is $\sfrac{\|\T{X} - \T{\tilde{X}}\|}{\|\T{X}\|}$.
However, the sampled least squares computations are so inexpensive
that checking this stopping condition
can take longer than the rest of the iteration. This is
particularly true for out-of-core problem sizes~\cite{VeLa16}.  
Thus, we propose a different sampling-based method for computing an
approximate stopping 
criterion and present a theoretical rationale for why it works.
One tempting strategy would be to track the norm of the residual within the sampled least squares computation. Unfortunately, the variance of this value is very high due to the small number of fibers sampled from $\T{X}$. Thus, we propose an alternative approach.

We use the notation $[N]$ to denote the set $\{1,\dots,N\}$. For a
given natural number $\hat P$, let 
\begin{displaymath}
  \mathcal{\hat I} \subset \mathcal{I} \equiv [I_1] \otimes [I_2] \otimes \cdots \otimes [I_N]
\end{displaymath}
be a uniform random subset of $\hat P$ indices of $\T{X}$. Let $\V{i}=(i_1, i_2,\dots,i_N)$ denote a multiindex, i.e., $x_{\V{i}} = x_{i_1i_2\cdots i_N}$.
Define $\T{E} = \T{X} - \T{\tilde{X}}$, and observe that
\begin{displaymath}
  \| \T{E} \|^2 = \sum_{\V{i} \in \mathcal{I}} e_{\V{i}}^2 
  = P \mu 
  \qtext{where} P = \prod_n I_n  
  \qtext{and}\mu = \text{mean} \set{e_{\V{i}}^2 | \V{i}\in\mathcal{I}}.
\end{displaymath}
We can approximate the mean $\mu$ with the mean
$\hat \mu$ of the subset of entries in $\mathcal{\hat I}$:
\begin{displaymath}
  \mu \approx \hat \mu 
  \qtext{where}
  \hat \mu = \text{mean} \set{e_{\V{i}}^2 | \V{i}\in\mathcal{\hat I}}.
\end{displaymath}
The relative residual norm can be estimated as
\begin{displaymath}
  \frac{\|\T{E}\|}{\|\T{X}\|}
  =
  \frac{(P \mu)^{1/2}}{\|\T{X}\|}  \approx
  \frac{(P \hat \mu)^{1/2}}{\|\T{X}\|}. 
\end{displaymath}
We can now apply the multiplicative Chernoff-Hoeffding bounds if we
make some assumptions on our data~\cite{rabook}. Assume the errors are
drawn from a finite distribution and $\mu$ is the
\emph{true mean} (or close enough to it). This is a reasonable
assumption is we assume that the CP models elicits the low-rank
structure which is contaminated by noise. We do not make assumptions
about what the specific distribution is.
Our samples are assumed to be i.i.d. 
Let $\mu_{\text{max}} = \max_{\V{i}}(e_\V{i}^2)$ be the maximum
allowable value. For any $\gamma > 0$
we have the following very conservative upper- and lower-tail bounds:
\begin{gather}
\label{eqn:chernoff}
\begin{split}
\Pr\{\hat \mu \geq (1+\gamma)\mu\} \leq \exp\left(-\frac{2\gamma^2 \mu^2 \hat P}{\mu_{\text{max}}^2}\right), \\
\Pr\{\hat \mu \leq (1-\gamma)\mu\} \leq \exp\left(-\frac{\gamma^2 \mu^2 \hat P}{\mu_{\text{max}}^2}\right).
\end{split}
\end{gather}
We can then write this in terms of the the residual norm:
\begin{lemma} For any $\gamma \in (0,1)$, we can bound the relative
  difference in the approximated and true error as
\begin{align*}
\Pr\left\{\sqrt{1-\gamma} \leq \frac{(P \hat \mu)^{1/2}}{\|\T{E}\|} \leq \sqrt{1+\gamma}\right\} \leq \exp\left(-2\frac{\gamma^2 \mu^2 \hat P}{\mu_{\text{max}}^2}\right).
\end{align*}
\end{lemma}
\begin{proof} Multiply both sides within the probability terms of~\cref{eqn:chernoff} by $P$, take the square root and simplify.
\end{proof}
Let $0 \leq \mu \leq 0.5$ (if the error is higher than this we are far
from terminating). Let $\sqrt{1+\gamma} = 1.05$; that is, we allow our
estimate to be wrong by $5\%$ multiplicatively. Then
$\gamma=0.1025$. A confidence of $98\%$ is maintained when $\hat P
\geq 372 \mu_{\text{max}}^2$. The fortunate aspect of this
conservative bound is that we expect $\mu_{\text{max}}$ to become
smaller and smaller as the ALS algorithm proceeds. In general, we 
usually assume $\mu_{\max} = 1$.

The cost of computing $\hat P \hat \mu$ to get an estimate of the error is $\inlineBigO{\hat P RN}$ flops.
For each sampled entry of the tensor, the corresponding entry of the model tensor must be computed via a sum of $R$ terms, each with $N+1$ multiplicands (including the weights).
For comparison, the cost of computing the exact error is $\inlineBigO{R\prod_n I_n}$ flops.

The relative residual of CP-ALS is guaranteed to decrease at each iteration, making a termination condition easy to specify (stop when the change in error drops below a threshold). 
Termination of CPRAND is more complicated because neither the true nor approximate error are guaranteed to decrease at each iteration. In practice, the simplest strategy is to store the lowest relative error achieved so far, and terminate when a particular number of iterations have elapsed without any reduction in this minimum. 

\section{Experiments} \label{sec:experiments} 

\label{sec:expts}
We evaluate the performance of the randomized algorithms on both
synthetic and real-world data. The  synthetic experiments
enable us to generate tensors from known latent factors and thus
measure whether the ground truth is recovered. 
We also consider 
real data sets which are free of the simplifying
assumptions of synthetic data (e.g., Gaussian noise) and demonstrate
our algorithms' effectiveness in practice.  

All experiments are run on MATLAB R2016a using Tensor Toolbox v2.6 
\cite{TTB_Dense, stop1, TTB_Software}
on
an Intel Xeon E5-2650 Ivy Bridge 2.0 GHz machine with 32 GB of memory.   
The CP-ALS implementation we compare
against incorporates the recent optimizations of Phan et
al.~\cite{phan}, without which it would be several times slower. 

\subsection{Computational Time}
\label{sec:computational-time}
Our first experiments ignore the convergence of the randomized methods
and compare only the computational time for each iteration.
Although the randomized methods will typically require more iterations, these
experiments enable us to understand the difference in costs for the
least squares solves and the initialization.
We consider convergence and solution quality
in subsequent subsections. 

\Cref{fig:periterspeedup} shows how much cheaper each iteration of the
ALS algorithm is when using the randomized least squares solvers.
We consider third- and fifth-order tensors of various sizes, where the dimensions of all modes are the same.
We used a target rank $R=5$ in all experiments.
We sampled $S=90$ rows for the randomized methods, although the exact
number of rows makes little difference in runtime.
For each size,
we compute the mean time for 100 iterations over three tensors. 
The convergence checks are entirely omitted in the
computation and do not contribute to the timings. 
Since we do a fixed number of iterations, the initial guesses are
irrelevant.
We see that as the size increases, the relative speedup of the randomized algorithms also increases to as much as $50\times$ for order-3 tensors over $500\times$ for order-5 tensors. 
This is mainly due to the per-iteration computational cost of the randomized algorithms being $\inlineBigO{NRIS}$, where $N$ is the order of the tensor and $I$ is the size of each dimension, as derived in \cref{sec:rcpals-cost}.
For comparison, the cost of CP-ALS is $\inlineBigO{NRI^N}$ flops per iteration.

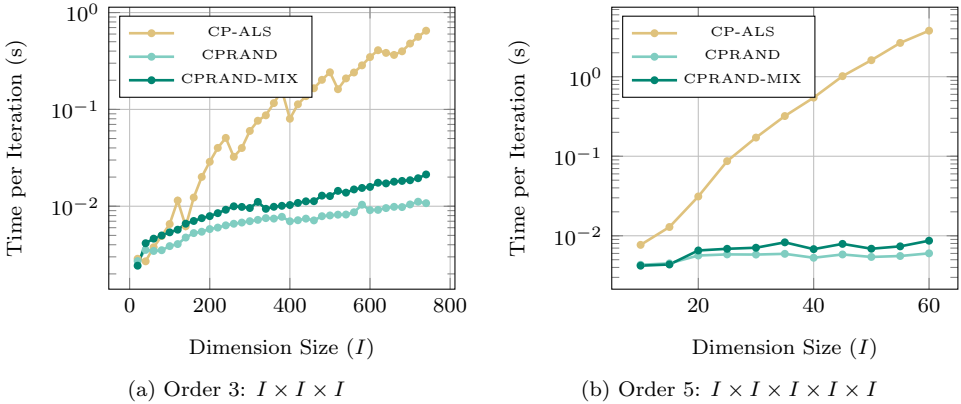
\begin{figure}[th]
  \centering 
  \subfloat[Order 3: $I \times I \times I$]{
    \begin{tikzpicture}
      \begin{semilogyaxis}[width=.475\textwidth, grid=major,
        xlabel={Dimension Size ($I$)}, 
        ylabel={Time per Iteration (s)}, 
        legend style={font=\tiny},
        legend entries={CP-ALS,CPRAND,CPRAND-MIX}, 
        legend pos=north west,
        ]
        \addplot[color=cpalsrcolor] table [x=dimcol,y=cptimes] {./tpi3-22-Dec-2016.dat};
        \addplot[color=cprandcolor] table [x=dimcol,y=cprandtimes] {./tpi3-22-Dec-2016.dat};
        \addplot[color=cprandfcolor] table [x=dimcol,y=cprandffttimes] {./tpi3-22-Dec-2016.dat};
      \end{semilogyaxis}
    \end{tikzpicture}
    \label{fig:periter3}} 
  \hfill
  \subfloat[Order 5: $I \times I \times I \times I \times I$]{
    \begin{tikzpicture}
      \begin{semilogyaxis}[width=.475\textwidth, grid=major,
        xlabel={Dimension Size ($I$)}, 
        ylabel={Time per Iteration (s)}, 
        legend style={font=\tiny},
        legend entries={CP-ALS,CPRAND,CPRAND-MIX}, 
        legend pos=north west,
        ]
        \addplot[color=cpalsrcolor] table [x=dimcol,y=cptimes] {./tpi5-22-Dec-2016.dat};
        \addplot[color=cprandcolor,mark=.] table [x=dimcol,y=cprandtimes] {./tpi5-22-Dec-2016.dat};
        \addplot[color=cprandfcolor,mark=.] table [x=dimcol,y=cprandffttimes] {./tpi5-22-Dec-2016.dat};
      \end{semilogyaxis}
    \end{tikzpicture}
    \label{fig:periter5}} 
  \caption{Mean time per iteration of CP-ALS, CPRAND and CPRAND-MIX
    for 3rd- and 5th-order tensors. 
    The target rank is $R=5$.
    The randomized methods use $S=90$ sampled rows.
    Each dot represents the mean iteration time for three different tensors
    over 100 iterations (no checks for convergence).
  }  
  \label{fig:periterspeedup}
\end{figure}

\cref{fig:estfitspeedup} demonstrates how much faster it is
to compute the stopping criterion based on the sampling approach described in \cref{sec:stopping}.
In this experiment, we compute the model fit both exactly and using $\hat P=2^{14}$ samples for order-3 and order-5 tensors.
The tensors and models are generated synthetically, with a prescribed fit of $95\%$.
For these problems, the largest relative error between the true and
sampled fit is less than $10^{-3}$, which is more than enough accuracy to make the correct decision on when to stop the iteration.
For both order-3 and order-5 tensors, we see speedups of over two orders of magnitude for the largest problems, and we can expect larger speedups for larger problems because the number of flops required by the sampling method is independent of the tensor dimensions.

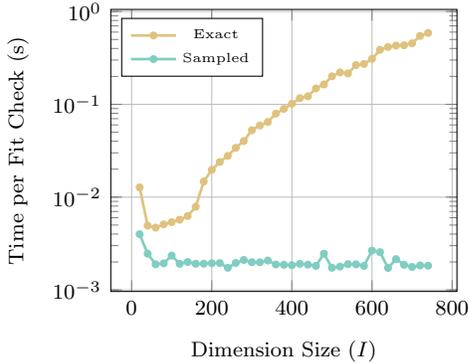
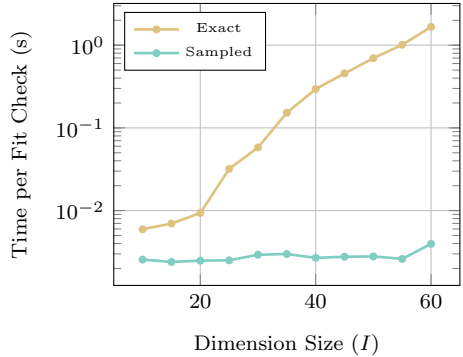
\begin{figure}[t]
  \centering 
  \subfloat[Order 3: $I \times I \times I$]{
    \begin{tikzpicture}
      \begin{semilogyaxis}[width=.475\textwidth, grid=major,
        xlabel={Dimension Size ($I$)}, 
        ylabel={Time per Fit Check (s)}, 
        legend style={font=\tiny},
        legend entries={Exact, Sampled}, 
        legend pos=north west,
        ]
        \addplot[color=cpalsrcolor] table [x=dimcol,y=cptimes] {./tpfc3-05-Jan-2017.dat};
        \addplot[color=cprandcolor,mark=.] table [x=dimcol,y=cprandtimes] {./tpfc3-05-Jan-2017.dat};
%        \addplot[color=cprandcolor] table [x=dimcol,y expr=\thisrow{cptimes} / \thisrow{cprandtimes}] {./tpfc3-05-Jan-2017.dat};
      \end{semilogyaxis}
    \end{tikzpicture}
    \label{fig:estfits3}} 
  \hfill
  \subfloat[Order 5: $I \times I \times I \times I \times I$]{
    \begin{tikzpicture}
      \begin{semilogyaxis}[width=.475\textwidth, grid=major,
        xlabel={Dimension Size ($I$)}, 
        ylabel={Time per Fit Check (s)}, 
        legend style={font=\tiny},
        legend entries={Exact, Sampled}, 
        legend pos=north west,
        ]
        \addplot[color=cpalsrcolor] table [x=dimcol,y=cptimes] {./tpfc5-05-Jan-2017.dat};
        \addplot[color=cprandcolor,mark=.] table [x=dimcol,y=cprandtimes] {./tpfc5-05-Jan-2017.dat};
%        \addplot[color=cprandcolor,mark=.] table [x=dimcol,y expr=\thisrow{cptimes} / \thisrow{cprandtimes}] {./tpfc5-05-Jan-2017.dat};
      \end{semilogyaxis}
    \end{tikzpicture}
    \label{fig:estfits5}} 
  \caption{Termination criterion (fit check) time of exact (for CP-ALS) and sampled (for CPRAND and CPRAND-MIX) for 3rd- and 5th-order tensors. 
    The target rank is $R=5$.
    The estimate of the fit is computed using $\hat P=2^{14}$ sampled entries.
    Each dot is the mean time to compute fit over 10 trials.
  }  
  \label{fig:estfitspeedup}
\end{figure}

We also consider the initialization costs of the methods.
For CPRAND-MIX, we need to apply an FFT to the fibers of the tensor in each mode.
If we use the HOSVD initialization for CP-ALS, then we need to
consider its cost. We note that we do not actually compute the full
HOSVD, but rather the leading left singular vectors of the unfolded
tensor $\M{X}_{(n)}$ for $n=2,\dots,N$.
\Cref{fig:mixtime} shows the preprocessing time for third- and
fifth-order tensors of various sizes, again where all modes have the same dimension.
The target rank (needed for HOSVD) is $R=5$.
Each data point is the mean time over 3 trials of 100 iterations each for the given size.
The cost of preprocessing for CPRAND-MIX is $\inlineBigO{NI^N\log I}$, as derived in \cref{sec:rcpalsfft-cost}, while the cost of HOSVD to initialize CP-ALS(H) is $\inlineBigO{NI^{N+1}}$ (we use an iterative eigenvector solver on the Gram matrix of each mode). 
While the cost of HOSVD is generally larger than mixing the tensor
with FFTs, the data access patterns of the methods are similar, and we
observe very similar timing results.
We mention that CP-ALS requires a good starting point, like HOSVD, to
ensure good performance. However, the randomized methods gain no
advantage from using the HOSVD initial guess so they use a random
initialization. 

\begin{figure}[th]
  \centering 
  \subfloat[Order 3: $I \times I \times I$]{
    \begin{tikzpicture}
      \begin{semilogyaxis}[width=.475\textwidth, grid=major,
        xlabel={Dimension Size ($I$)}, 
        ylabel={Preprocessing Time (s)}, 
        legend style={font=\tiny},legend entries={HOSVD,MIX}, legend pos=south east,
        ]
        \addplot[color=cpalsncolor,mark=.] table [x=Var1,y=Var2] {./3dmixtimesfftvsnvecs.dat};
        \addplot[color=cprandcolor,mark=.] table [x=Var1,y=Var3] {./3dmixtimesfftvsnvecs.dat};
      \end{semilogyaxis}
    \end{tikzpicture}
    \label{fig:mixtime3d}} 
  \hfill
  \subfloat[Order 5: $I \times I \times I \times I \times I$]{
    \begin{tikzpicture}
      \begin{semilogyaxis}[width=.475\textwidth, grid=major,
        xlabel={Dimension Size ($I$)}, 
        ylabel={Preprocessing Time (s)}, 
        legend style={font=\tiny},
        legend style={font=\tiny},legend entries={HOSVD,MIX}, legend pos=south east,
        ]
        \addplot[color=cpalsncolor,mark=.] table [x=Var1,y=Var2] {./5dmixtimesfftvsnvecs.dat};
        \addplot[color=cprandcolor,mark=.] table [x=Var1,y=Var3] {./5dmixtimesfftvsnvecs.dat};
      \end{semilogyaxis}
    \end{tikzpicture}
    \label{fig:mixtime5d}} 
  \caption{Initialization time comparison between MIX (for CPRAND) and
    HOSVD (for CP-ALS) for 3rd- and 5th-order tensors. 
    The target rank is $R=5$ for HOSVD. Each dot represents the mean of
    3 trials.}
  \label{fig:mixtime}
\end{figure}
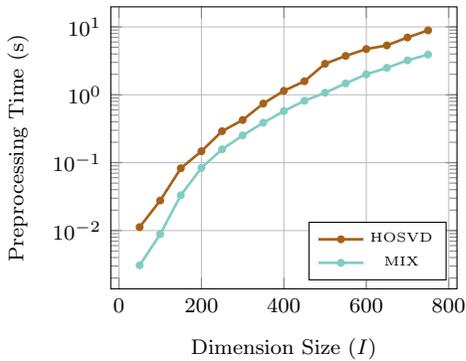
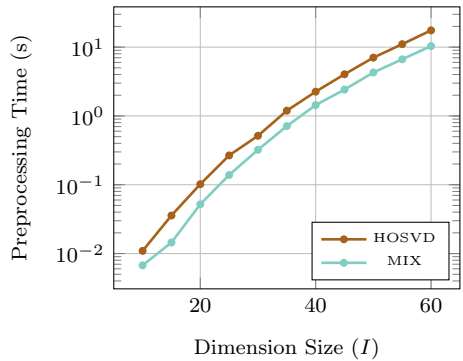

\subsection{Synthetic Data} \label{sec:datagen}
For our experiments on synthetic tensors, we use various
generation parameters. 
We create tensors based on known randomly-generated weight vectors
($\bm{\lambda}$) and factor matrices ($\set{ \Mn{A}{n} }$). In this
way, we know the true solution.
We consider 3rd and 4th-order problems, i.e., $N \in \set{3,4}$.
In the 3rd-order case, we set the size to be $400 \times 400 \times
400$, and in the 4th-order case we set the size to be $90 \times 90
\times 90 \times 90$.
For all experiments, we set the rank to be $R_{\text{true}}=5$.
The weight vector $\bm{\lambda} \in \mathbb{R}^{R_{\text{true}}}$ has
entries drawn uniformly from $[0.2,0.8]$.
The factor matrices $\Mn{A}{n} \in \mathbb{R}^{I_n \times R_{\text{true}}}$ are randomly generated as described in
\cite{ToBr06} so that the columns have collinearity $C$, which means
that any two column vectors from the same factor matrix satisfy
\begin{displaymath}
C = \frac{\V{a}_r^{(n)\trans}\V{a}_s^{(n)}}{\| \V{a}_r^{(n)\trans} \| \| \V{a}_s^{(n)}\|}.
\end{displaymath}
Intuitively, higher collinearity corresponds to greater overlap
between factors, while geometrically it corresponds to smaller angles
between factor vectors. High collinearity makes the original factors
harder to recover using CP-ALS, and can introduce swamping
behavior~\cite{swamps1}. In our experiments we generate tensors with
$C \in \set{0.5,0.9}$.
Using the synthetic weight vectors ($\bm{\lambda}$) and factor
matrices ($\set{ \Mn{A}{n} }$), the tensor we are trying to recover is
\begin{displaymath}
    \T{X}_{\text{true}} = 
  \sum_{r=1}^{R_\text{true}} 
  \lambda_r \;
  \V{a}_r^{(1)} \circ \V{a}_r^{(2)} \circ \cdots \circ \V{a}_r^{(N)}.
\end{displaymath}
Finally, we add noise to obtain the observed tensor. 
Let $\T{N} \in \mathbb{R}^{I_1\times I_2 \times \dots \times I_N}$ be
a noise tensor with entries drawn from a standard normal
distribution. Then our observed tensor is
\begin{equation*}
  \label{eqn:synth}
  \T{X} =  \T{X}_{\text{true}} + 
  \eta\left(\frac{\|\T{X}_{\text{true}}\|}{\|\T{N}\|}\right) \T{N},
\end{equation*}
where the parameter $\eta \in \set{0.01,0.10}$ is the amount of noise.
Generally, the rank is unknown, so 
we run the algorithms with $R \in \set{R_{\text{true}},
R_{\text{true}}+1}$.
The parameters for the experiments are summarized in \cref{tab:parms}.

\begin{table}[htbp!]
\smaller
  \centering
  \caption{Parameters varied in synthetic experiments.}
	\begin{tabular}{|c c|}
	    \hline
	    \textbf{Parameter} & \textbf{Values} \\ \hline
	    Order \& Size ($N, I$) & $\set{(3,400),(4,90)}$ \\ \hline
	    True \# Components ($R_{\text{true}}$) & $5$ \\ \hline
	    Collinearity ($C$) &  $\set{0.5,0.9}$  \\ \hline
	    Noise ($\eta$) &  $\set{0.01,0.1}$  \\ \hline
	    Model \# Components ($R$) & $\set{5,6}$ \\ \hline 
	  \end{tabular}
  \label{tab:parms}
\end{table}

We use the standard \emph{score} metric to measure how well the ground
truth is recovered by a CP decomposition in those cases where the true
factors are known \cite{ToBr06}. The score between two rank-one tensors
$\T{X}=\V{a}\circ\V{b}\circ\V{c}$ and
$\T{Y}=\V{p}\circ\V{q}\circ\V{r}$ is defined as:
\begin{equation}\label{eq:score}
\text{score}(\T{X},\T{Y}) = \frac{\V{a}^\trans\V{p}}{\|\V{a}\|\|\V{p}\|} \times \frac{\V{b}^\trans\V{q}}{\|\V{b}\|\|\V{q}\|} \times \frac{\V{c}^\trans\V{r}}{\|\V{c}\|\|\V{r}\|}.
\end{equation}
The $\bm{\lambda}$ values are ignored.
For $R>1$, we average the scores for all pairs of components.  Alas, the CP decomposition does not recover factors in their original order, so the score is the maximal average across all permutations of rank-one components.

We use the \emph{fit} to determine convergence, and the fit is defined as
\begin{displaymath}
  F = 1 - \frac{ \|\T{X} - \T{\tilde X} \|}{\|\T{X}\|}.
\end{displaymath}
For a tensor with $\eta$ noise, we expect the final fit to be at best
$1-\eta$.
In the presence of noise, maximizing fit does not exactly correspond to maximizing score.
The method
terminates when either the number of iterations exceeds 200, the
change in fit goes below $10^{-4}$ ($|F_{t} - F_{t-1}| \leq
10^{-4})$, or the fit is within $20\%$ of the noise level $(F_t \geq 1 -
1.2\eta)$.
All methods use the same criteria for termination with the exception
that CPRAND and CPRAND-MIX use an approximation to the fit, as discussed
in~\cref{sec:stopping}.  Specifically, these methods compute the error at $\hat
P$ entries and use that to estimate the overall fit.
We use $\hat P=2^{14}$ in these experiments, and 
we stress that the same $\hat P$ entries are used across all iterations.

As mentioned previously, we use the CP-ALS method provided by the
Tensor Toolbox for MATLAB. 
For CPRAND and CPRAND-MIX, the number of rows sampled for each least
squares solve is $S=80$ for
$R=5$ and $S=108$ for $R=6$. 
Here we stress that we make a new random
selection of rows at each iteration.

\newcommand{\boxplotcaption}[1]{%
  Results on 200 synthetic tensors of
  size #1 with $R_{\text{true}}=5$ and collinearity
  $C \in \set{0.5,0.9}$ in the factors. 
  We stop when the number of iterations exceeds 200, 
  the fit stagnates $|F_t - F_{t-1}| \leq 10^{-4}$, or
  the fit exceeds a preset threshold $F_t \geq 1-1.2\eta$.
  Each of the four methods is tested 
  with target ranks $R \in \set{5,6}$
  and three random starts, except for CP-ALS (H) which as the one
  fixed start. We report results from all starts (2000 in all).
  For CPRAND and CPRAND-MIX,
  we use $\hat P = 2^{14}$ random entries to check convergence and
  the number of row samples in the least squares solve is
  $S=80$  for $R=5$ and $S=108$ for $R=6$.}

\begin{figure}[tbhp]
  \centering 
  \subfloat[Noise $\eta = 1\%$]{
    \centering 
    \includegraphics[width=0.95\linewidth]{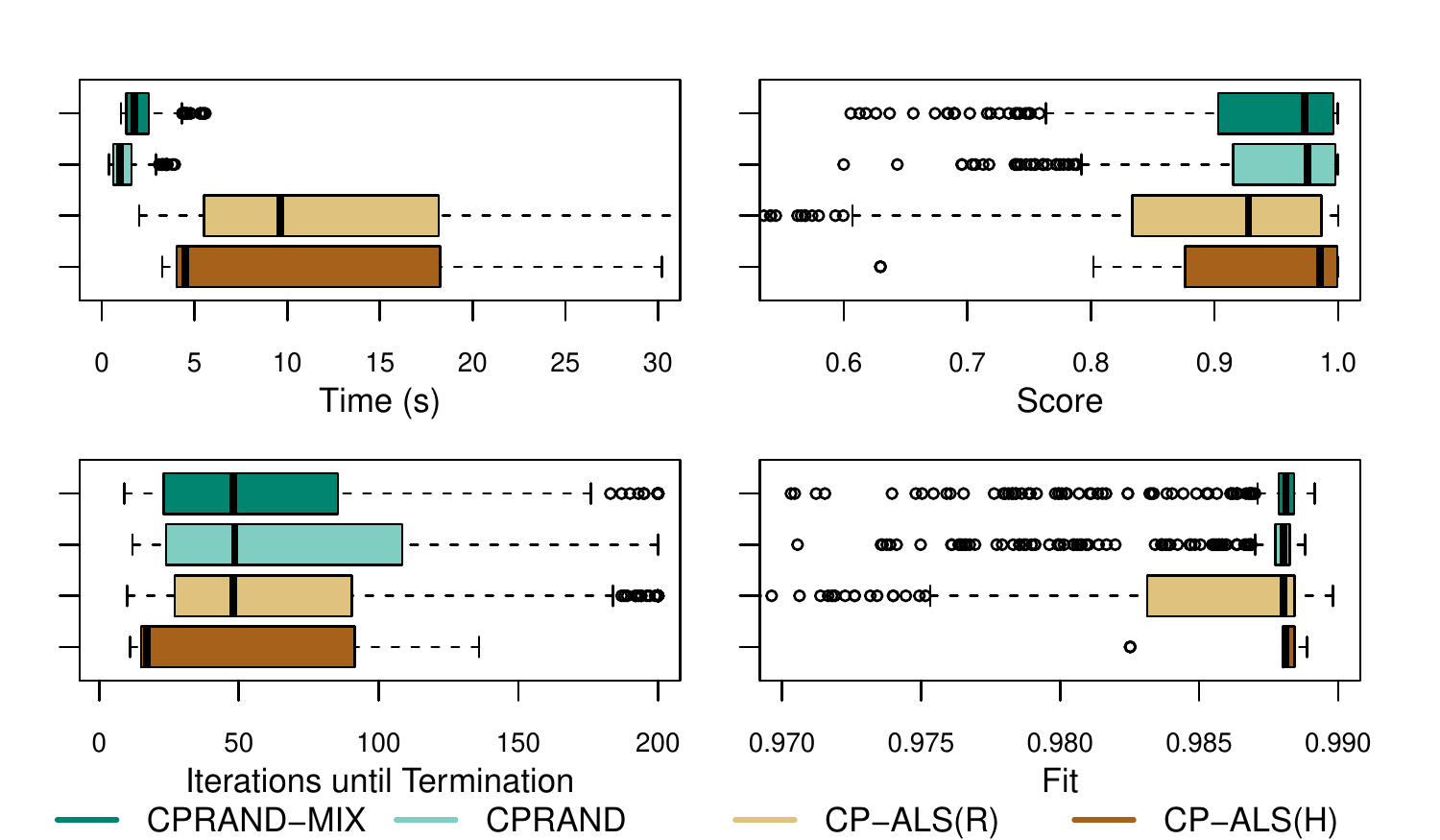}
    \label{fig:400plotNS01R5}
  }\\ \vspace{-0.2in}
  \subfloat[Noise $\eta = 10\%$]{
    \centering 
    \includegraphics[width=0.95\linewidth]{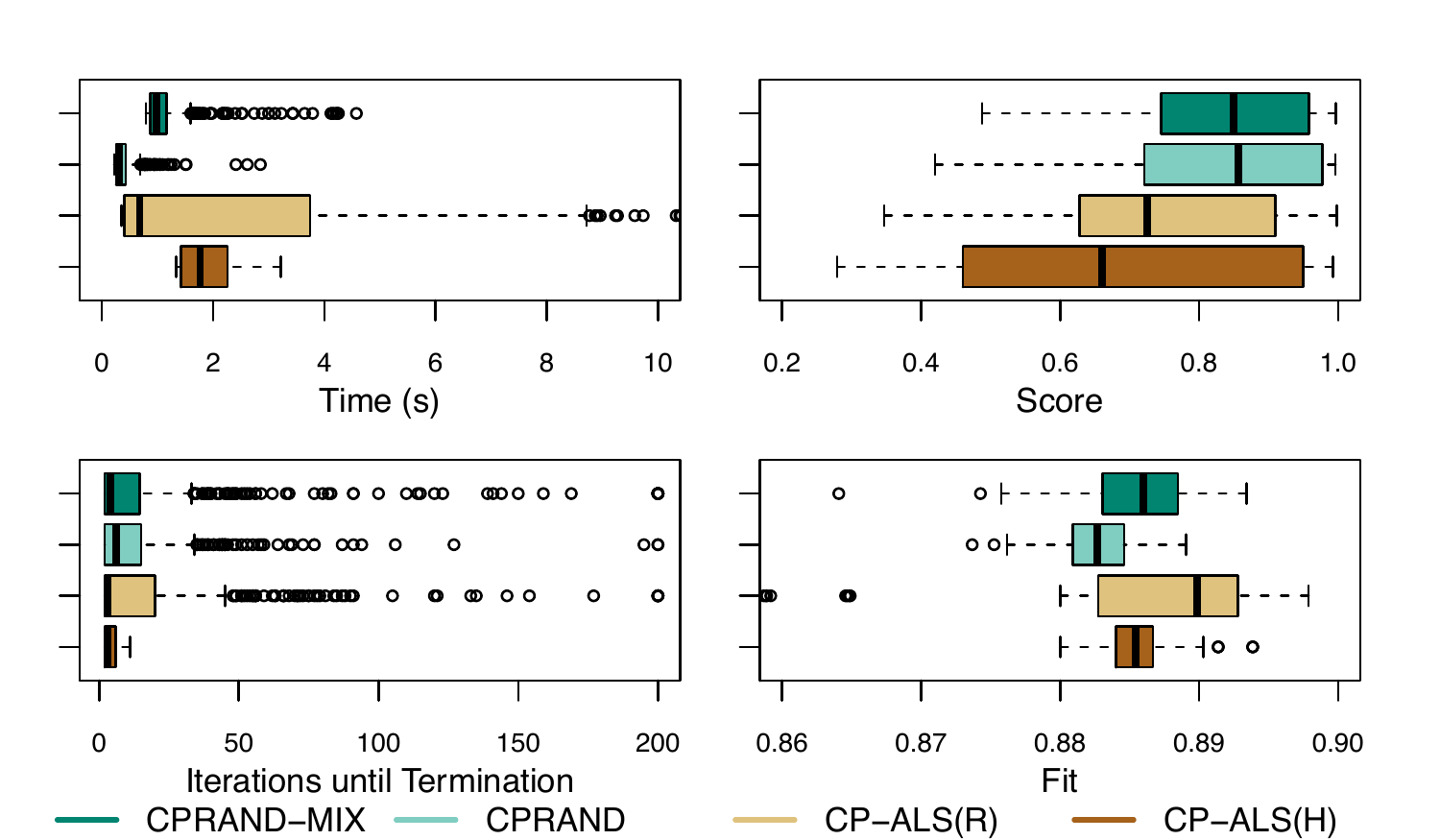}
    \label{fig:400plotNS10R5}
  }
  \caption{\boxplotcaption{$400\times400\times400$}}
  \label{fig:syn3}
\end{figure}

\begin{figure}[tbhp]
  \centering 
  \subfloat[Noise $\eta = 1\%$]{
    \centering 
    \includegraphics[width=0.95\linewidth]{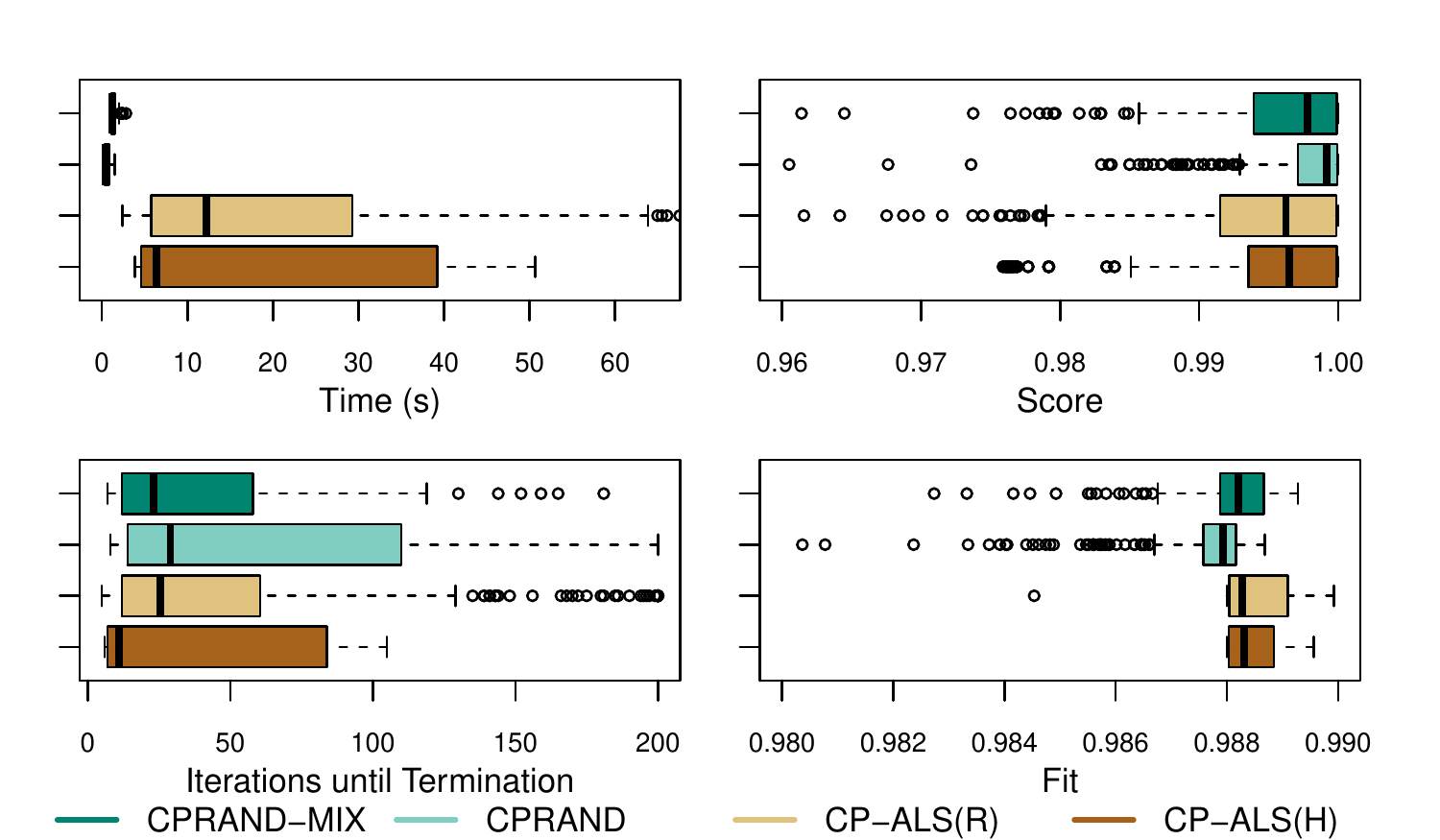}
    \label{fig:90plotNS01}
  }\\ \vspace{-0.2in}
  \subfloat[Noise $\eta = 10\%$]{
    \centering 
    \includegraphics[width=0.95\linewidth]{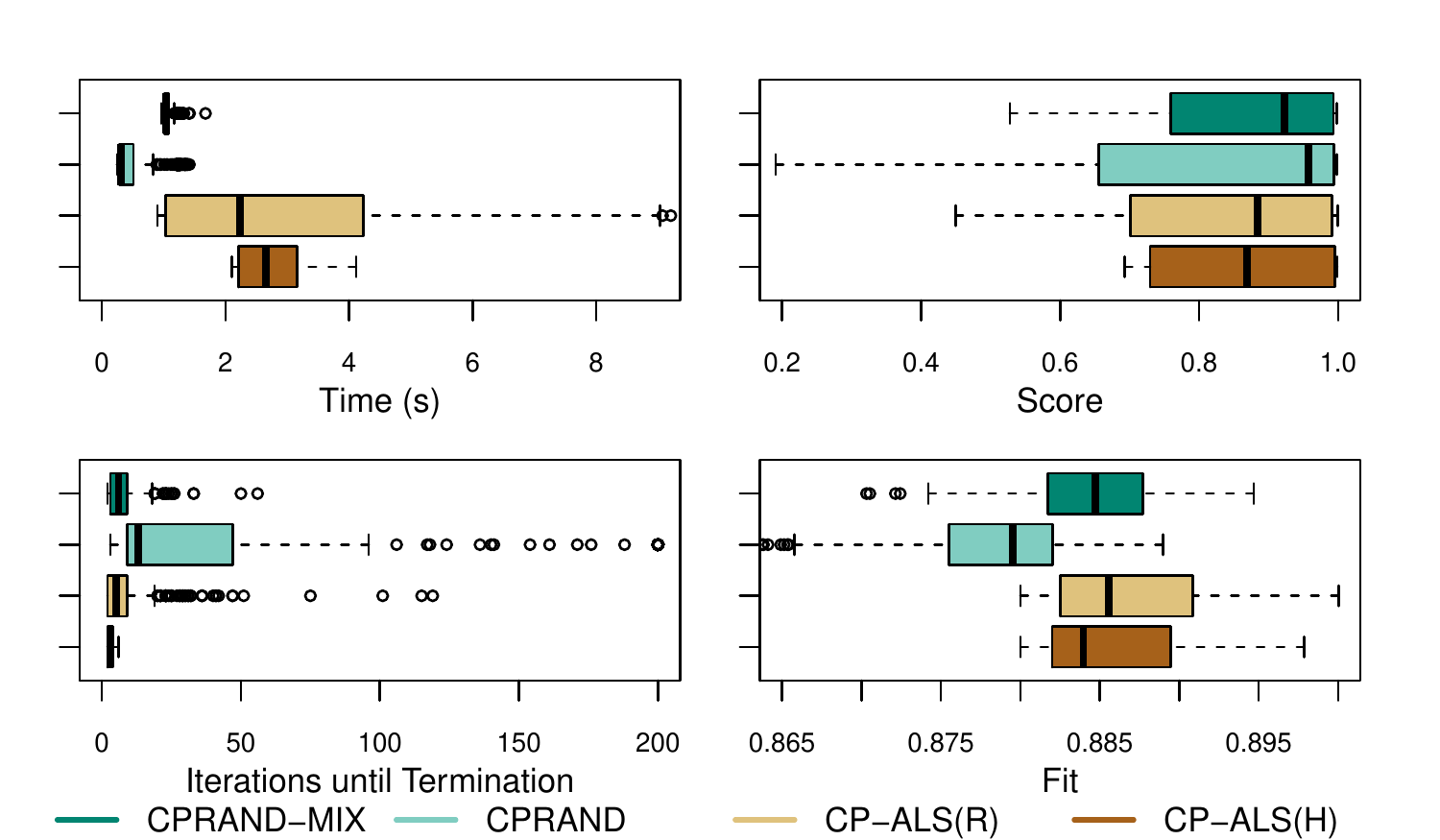}
    \label{fig:90plotNS10}
  }
  \caption{\boxplotcaption{$90\times90\times90\times90$}}
  \label{fig:syn4}
\end{figure}

For each possible combination of tensor order/size, collinearity, and
noise level (listed in~\cref{tab:parms}), we generate 50 synthetic
tensors. 
In \cref{fig:syn3,fig:syn4},
we show box plots comparing CPRAND and CPRAND-MIX with random initialization versus
CP-ALS with HOSVD (H) and random (R) initialization. 
The box plots show a box that indicates the 25th-75th quartiles, and
the median is indicated by a vertical line inside the box. Outliers
are displayed as circles.
Each subfigure shows the results on 100 distinct tensors, i.e., fixed
order/size and noise level with 50 tensors each for $C \in \set{0.5,0.9}$.
We test each method with $R \in \set{5,6}$.
For random initialization, we have three starting points. 
Therefore, each row in the box plot is the result of 600 runs for
CPRAND, CPRAND-MIX, and CP-ALS~(R) and 200 runs for CP-ALS~(H).
For each trial we measure the time, number of outer iterations, fit,
and score at termination.  

First, we consider the quality of the solutions in terms of the fit
(always in the range $[0,1]$),
i.e., the objective function being maximized.
This is shown in the lower right plot of each subfigure, and we report
the true final fits even if approximate fits are used in the algorithm.
The median fits are essentially
identical, with a maximum difference of 0.006. The CP-ALS (R) has the
highest variance in the fit since it is highly dependent on the
quality of the starting point. The CPRAND and CPRAND-MIX have less
variance in their fits.
In general, the CPRAND methods are much less sensitive to starting
point --- so much so that we do not even include results using the
HOSVD initialization.  

Second, because these problems are synthetic, we know the true
underlying factors and so can use the score from \cref{eq:score}, 
with 1.00 being a perfect match.
The scores are shown in the upper right plot of each subfigure.
We note that we see little difference between CPRAND and
CPRAND-MIX in terms of quality because these artificially generated
problems do not have high coherence.
At low noise (1\%), the methods all do well, with medians above 0.97
for CPRAND, CPRAND-MIX, and CP-ALS(H) in
\cref{fig:400plotNS01R5} and above 0.99 for all methods in
\cref{fig:90plotNS01}.
We see a striking difference, however, at the 10\% noise level. The
medians for CPRAND and CPRAND-MIX are above 0.85, whereas CP-ALS are below
0.73 in \cref{fig:400plotNS10R5}.
Similarly, 
medians for CPRAND and CPRAND-MIX are above 0.92, whereas CP-ALS are below
0.89 in \cref{fig:90plotNS10}.
We contend that the randomized methods are more robust 
because they avoid the problem of overfitting to noise thanks
to the randomization.
This is evidenced by the fact that all methods achieve similar fit but the randomized methods tend to achieve better score.

Third, we look at the number of iterations, shown in the lower left
plot of each subfigure.
The median number of iterations for the randomized methods are always
higher than CP-ALS (H), since it has the advantage of a good starting
point.
This is because the randomized methods generally make less progress
per iteration.
With the exception of CP-ALS (H), every method hit the maximum number
of iterations (200) at least once.

Fourth, we consider runtime, where we expect to get the most benefit.
This are shown in the upper left plot of each subfigure.
In the low noise (1\%) case, we see an improvement in median runtime of
10X for CPRAND versus CP-ALS (R) in \cref{fig:400plotNS01R5}
and 30X for the same pair in \cref{fig:90plotNS01}.
The cost of the preprocessing for CPRAND-FFT means it may be about 3X
slower than CPRAND without mixing, depending on the number of
iterations.
In the high noise (10\%) case, the difference in time is less
dramatic, but we have the improvement in scores discussed above.

In summary, the synthetic results suggest that the CPRAND and
CPRAND-MIX methods produce solutions that are at least as good and
sometimes much better than 
CP-ALS in terms of quality (fit and score).
Moreover, the randomized algorithms are at least as fast as the
standard methods and sometimes much faster.

\subsection{COIL Data Set}
\label{sec:coil-data-set}
COIL-100 is an image-recognition data set that contains images of
objects in different poses~\cite{coil100} and has been used previously
by Zhou, Cichocki, and Xie~\cite{bigtens} as a tensor decomposition
benchmark.
The problem is set up as follows.
There are 100 different object classes, each of which is imaged from
72 different angles. Each image is sized to $128 \times 128$ pixels in
three color channels (RGB). If we discard the ground truth, 
we have a $128 \times 128
\times 3 \times 7200$ tensor of size 2.8GB. The
irregular dimensions and large size of the data make for an
interesting CP benchmark. 

In our experiment, we compare the runtimes of CP-ALS (R) and
CPRAND-MIX. Unlike the synthetic experiments, this experiment required mixing to converge for a reasonable number of samples.
We use $R=20$ factors. 
We ran five trials of CP-ALS and terminated when the change in fit
went below $10^{-4}$. This yielded a median runtime of 204 seconds and
a fit of $0.686$.
For CPRAND-MIX, we terminate when the fit fails to improve for five
consecutive iterations and 
vary the number of samples ($\samplesize$).
For each $\samplesize$, we run five trials with random starting points.
We compute the approximate fit for CPRAND-MIX with sample size $\hat
P=2^{14}$.
We vary $S$ and show the results in \cref{tab:coil}. The fits are very
close to the fit obtained by CP-ALS, with speedups as high as
8$\times$.
The speedup does not decrease monotonically with $S$, since
differences in the number of iterations to converge may have some
impact. Nevertheless, increasing the number of samples increases
the cost per iteration and thus reduces the overall speedup on average.

\begin{table}[t]
  \centering\footnotesize
  \sisetup{round-mode=places}
  \caption{CPRAND-MIX speedup and accuracy on COIL tensor of size $128 \times 128
    \times 3 \times 7200$. Reporting median runtimes over five trials with
    random starting points, $R=20$ components, $\hat P = 2^{14}$ entries
    for approximate fit, and varying number of samples $S$. Speedup
    compared against median runtime of CP-ALS over five trials with random starting points.}
  \label{tab:coil}
  \begin{tabular}{|S[table-format=4.0]|S[round-precision=2,table-format=.2]|S[round-precision=3,table-format=.3]|}
    \hline
    \multicolumn{1}{|c|}{\bf \# Samples ($S$)} &
    \multicolumn{1}{c|}{\bf Speedup} &
    \multicolumn{1}{c|}{\bf Fit} \\ 
    \hline
    400 & 8.37821245976872 & 0.674427373291553 \\ 
    450 & 7.97986205183953 & 0.676401067180041 \\ 
    500 & 6.63260244165385 & 0.677139575576137 \\ 
    550 & 7.28819766345918 & 0.677671827719428 \\ 
    600 & 4.75391554510243 & 0.680004470176377 \\ 
    650 & 4.72795654513824 & 0.679898697712541 \\ 
    700 & 4.77492964294695 & 0.680399654058832 \\ 
    750 & 4.52011688273789 & 0.680905539528707 \\ 
    800 & 3.70307615304456 & 0.681763846278681 \\ 
    850 & 4.89505864379786 & 0.678200332500891 \\ 
    900 & 4.94563231318697 & 0.679038846110242 \\ 
    950 & 4.22420843976262 & 0.682181519741705 \\ 
    1000 & 2.8424227150739 & 0.683819845468544 \\     
    \hline
    \multicolumn{1}{|c|}{CP-ALS} & 1.00 & 0.686 \\
    \hline
  \end{tabular}
\end{table}

\subsection{Hazardous Gases}
\label{sec:hazardous-gases}
Vervliet and De Lathauwer~\cite{VeLa16} demonstrate their randomized
block sampling approach for CP on a hazardous gas classification
task~\cite{gas}, so we
compare on the same dataset.
The data comes from 899 experiments (actually 900 experiments, but one
is omitted) where one of three different hazardous gases (carbon monoxide, acetaldehyde, or ammonia) is released into
a wind tunnel and its concentration is measured across 72 sensors for
25,900 time steps. 
We use the exact same preprocessing script as Vervliet and De
Lathauwer~\cite{VeLa16}: missing values are interpolated and the 
data is normalized, cropped, and centered. The resulting tensor of 
size $25{,}900\times72\times899$ requires 13.4~GB storage. 
The first mode corresponds to the 25,900 time steps,
the second mode corresponds to the 72 sensors, and
the third mode corresponds to the 899 experiments.
We compute the
CP decomposition with $R=5$ (as in \cite{VeLa16}).
We run CP-ALS with both random (R) and HOSVD (H) initialization.
Due to the size of the tensor, we run CPRAND without mixing, using
random initialization, a
sample size of $\samplesize=1000$ rows per least squares solve, and
using $\hat P=2^{14}$ entries for the stopping condition.
The ALS methods terminate when
change in fit goes below $10^{-4}$ (i.e., $|F_{t} - F_{t-1}| \leq
10^{-4})$, and CPRAND terminates when estimated fit fails to improve after ten iterations. 
We run 10 trials each of CPRAND and CP-ALS(R), and a single trial of CP-ALS(H).
The median run times are listed in \cref{tab:gases}, and we can see that the median time for
CPRAND is less than one minute.
CPRAND was nearly 4$\times$ faster than CP-ALS(R) and achieves roughly the same classification error.
CPRAND is 10$\times$
faster than CP-ALS(H), which incurs a high initialization cost. 
We cannot compare runtimes with Vervliet and De~Lathauwer~\cite{VeLa16} since they are on a different
computational architecture,
but they report a runtime of less than three minutes for their method, which is in the same ballpark.

We next consider the quality of the decomposition.  Vervliet and De
Lathauwer manually selected three factors (column vectors) from the
experiment factor matrix and used those to classify the experiments
according to which of the three gases was released. We do a similar
experiment, except rather than choosing the three vectors manually, we
tried all ten (five choose three) possible choices of three vectors
and report on the best one.
The rows of this sub-factor matrix can then be thought of as
three-dimensional points. We run $k$-means on these points and measure
the classification error, which is the percentage of the 899
experiments that are misclassified. For each trial we performed a single run of
$k$-means with random initialization.
The median fits and classification errors over all trials are shown in \cref{tab:gases}. Both ALS methods achieve a median
classification error of 0.67\% (6 misclassified), while CPRAND achieves a marginally better classification error of 0.61\% (5.5 misclassified) despite a slightly lower fit value. For comparison, Vervliet and De Lathauwer~\cite{VeLa16}
report classification errors of 0.3--0.8\% for 100 runs of their randomized block
sampling approach, with the addition of a specialized step criteria; 
without the specialized step, their performance degrades to 5\%
error. 

\begin{table}[t]
\smaller
  \centering
  \caption{Results over ten trials on a $25{,}900 \times 72 \times 899$ tensor
    representing experiments with three hazardous gases.
    We use $R=5$ factors, and stop when the fit
    stagnates, i.e., $|F_{t} - F_{t-1}| \leq  10^{-4}$ for CP-ALS, or when estimated fit fails to improve after 10 iterations of CPRAND.
    For CPRAND, we use $\hat P = 2^{14}$ random entries to check
    convergence and $S=1000$ row samples in each least squares
    solve. 
    We run $k$-means on three columns of the experiment factor matrix
    with a target of three clusters. Using the $k$-means output we report the
    median proportion of 899 experiments that are misclassified (according to which
    hazardous gas was used).}
	\begin{tabular}{|l c c c|}
	    \hline
	    \textbf{Method} & \textbf{Median Time} (s) & \textbf{Median Fit} & \textbf{Median Classification Error} \\ \hline
	    CPRAND & $53.6$ & $0.715$ & $0.61\%$ \\ \hline
	    CP-ALS (H) & $578.4$ & $0.724$ &  $0.67\%$ \\ \hline
	    CP-ALS (R) & $204.7$ & $0.724$ & $0.67\%$ \\ \hline
	  \end{tabular}
  \label{tab:gases}
\end{table}

We visualize the factors computed by CPRAND in \cref{fig:gases}. It is
easy to see that the results can be used to classify the gases. For
instance, the first factor clearly separates the purple gas from the
green and red. Similarly, the fourth factor clearly separates red from
the green and purple.
The fifth factor is the smallest
magnitude and appears to be a ``noise'' factor.

\begin{figure}[ht]
  \centering
  \includegraphics[width=0.98\textwidth]{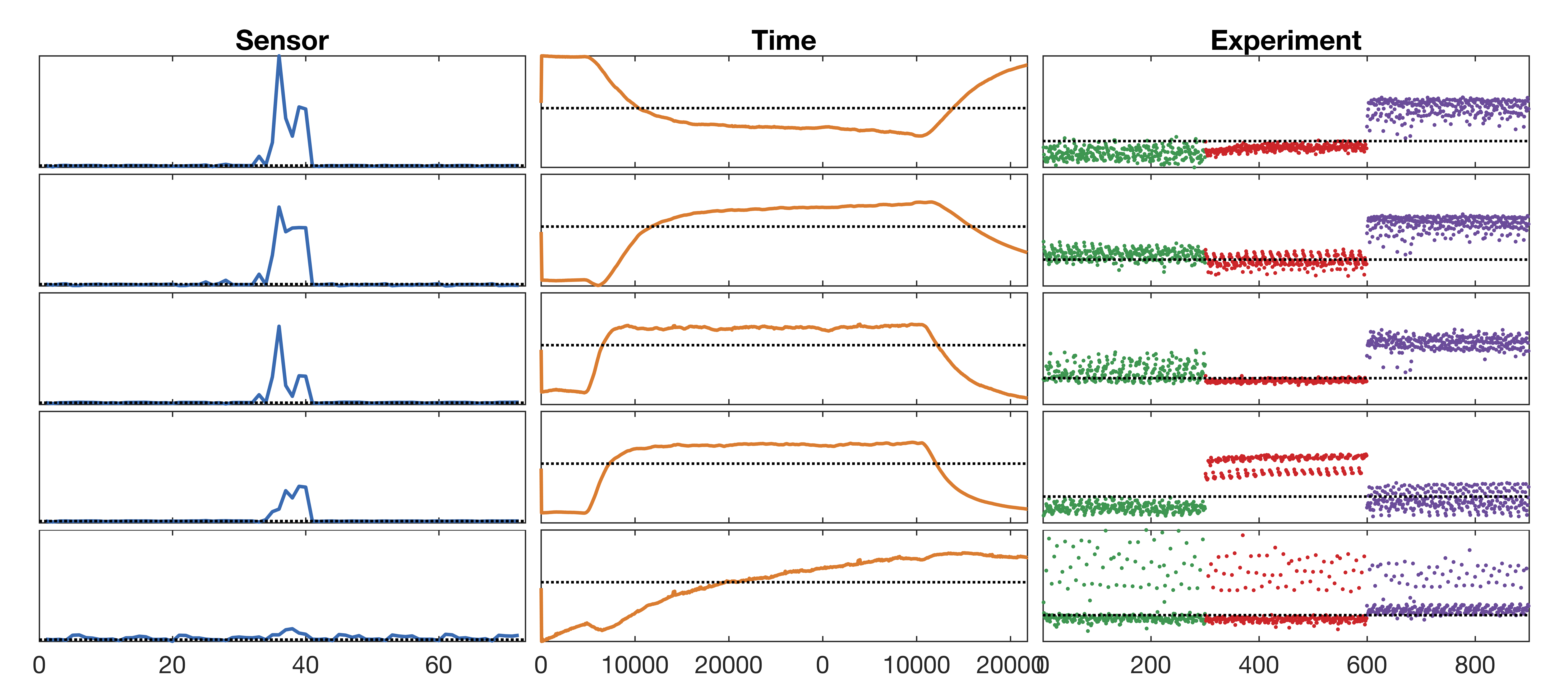}
  \caption{Visualization of the five factors from the $25900 \times
    72 \times 899$ hazardous gas tensor, as computed by CRAND. The
    factors are sorted by magnitude, from largest at the top to
    smallest at the bottom. The magnitude is reflected in the sensor
    factors (left). The time (middle) and experiment (right) factors
    are normalized to unit norm. The three gas types are color-coded
    in the experiments factors. All factors in the same mode are
    plotted on the same y-scale, and the dashed line is zero.}
  \label{fig:gases}
\end{figure}

\section{Related Work} \label{sec:related} 
Vervliet and Lauthauwer present a stochastic gradient descent (SGD) algorithm for CP that samples blocks from the original tensor to update corresponding blocks of the factor matrices~\cite{VeLa16}. This approach is similar in spirit to CPRAND, but takes an altogether different approach to the randomization. They use \emph{contiguous} samples in the block updates and finer control over step sizes. They also update only a portion of each factor matrix in each iteration.
An interesting contrast between the two methods can be seen by comparing the factors that are computed in the example from \cref{sec:hazardous-gases}.  Our factors are shown in \cref{fig:gases}. The time factors are much smoother than
the factors pictured in~\cite{VeLa16}. Since the block method 
updates only a small subset of each factor at a time, we conjecture that this may explain the
blockiness of the solution. 
Vervliet and Lauthauwer~\cite{VeLa16} also propose an inexpensive stopping condition based on Cramer-Rao bounds, but these require some additional knowledge about the noise level.
Another framework that draws from SGD is FlexiFaCT, which targets coupled tensor decompositions for parallel computation~\cite{flexifact}.

Cheng et al.~have recently applied a leverage score-based sampling to the least-squares step of the sparse CP decomposition by showing how leverage scores of an unfolded tensor can be estimated by the leverage scores of the factor matrices~\cite{spals}. This approach is similar to the way that we bound the coherence of Khatri-Rao products. 
Reynolds et al.~also use randomization within CP-ALS, specifically for the case of rank reduction, where the input to the algorithm is already in CP format \cite{RDB16}.
They use randomization to improve the conditioning of the individual least squares problems in order to compute better overall approximations.
The randomization makes each iteration of their method more costly, but they observe faster convergence (and overall running time) than ALS for ill-conditioned problems.

Wang et al.~have applied sketching methods to \emph{orthogonal} tensors with provable guarantees~\cite{anand}. Song et al.~show that this sketch can be computed without reading the entire tensor (in sublinear time) under certain conditions~\cite{sublinear}.

An alternative to sketching is to compress the tensor using lossy methods before computation. Zhou and Cichocki examine the effectiveness of performing a CP decomposition on a compressed representation of the data using the lossy Tucker decomposition to produce a smaller problem size~\cite{bigtens}. ParCube~\cite{parcube} compresses the original tensor by directly sampling and performs a decomposition on the result.

\section{Conclusion} \label{sec:conclusions} 

We have provided an example of the power of randomized methods in the
context of CP decompositions for dense tensors. 
As discussed in the related work (\cref{sec:related}), a few
approaches have been recently proposed. Ours is a unique approach that
focuses on the least squares subproblem. The advantage of this
approach is that we can leverage existing theory and methodology.
Specifically, we have a practical implementation, using MATLAB and the
Tensor Toolbox, that efficiently employs randomized least squares in
the context of CP-ALS.

The least squares subproblems have a special structure that allows for
very efficient solution; however, this still requires formation of the
Khatri-Rao matrix and multiplication with the matricized tensor, which is the primary computational bottleneck.
In our implementation of the randomized approach, we entirely avoid
forming the Khatri-Rao matrix and so greatly reduce the expense of the
least squares solve.
We refer to this method as CPRAND.

It is oftentimes a good idea to apply an FJLT to the least squares
problem to ensure incoherence. We explain how this can be done in a
preprocessing step rather than for every least squares solve. 
Assuming that we use an FFT in the FJLT, we have to reverse part of
the transformation for each linear solve. However, we need only apply
the inverse transform to the small sampled matrix at trivial cost.
We refer to this method as CPRAND-MIX.
A small disadvantage of the FFT is that it transforms a real-valued
problem to be complex-valued, doubling the memory requirement. 
The computational cost difference, however, is negligible.
On the other hand, we could use a real-valued transform, but these
proved to be slower than the FFT in MATLAB with no improvement in
the quality of the CP decomposition.

Checking the stopping condition is also a significant expense.
This is based on the fit of the 
model to the original data and so requires forming the Khatri-Rao
matrix, an expense we avoid in the least squares solves. 
We employ another type of randomization in this case, based on using
just a subsample of the tensor entries for comparison.
Assuming the errors are i.i.d. and drawn from a finite distribution,
then we can approximate the model fit error with reasonable accuracy
and substantially reduced cost, so we employ this stopping condition
in our randomized algorithms.

We have demonstrated the benefits of CPRAND and CPRAND-MIX in both
synthetic and real data experiments, including large-scale tensors of
up to 13~GB in size. The randomized methods are
overall much faster than CP-ALS for equivalent quality 
decompositions. Moreover, the CPRAND methods are much less sensitive to
the initial guess. We conjecture that the randomization prevents
getting stuck in a local minimum caused by overfitting the noise.
The CPRAND-MIX is more expensive to initialize (requiring the
application of an FFT in each mode) but has the advantage of ensuring
coherence. We found that mixing was critical for good performance on
the COIL-100 dataset in \cref{sec:coil-data-set} but unnecessary for
the hazardous gas dataset \cref{sec:hazardous-gases}.
The expense of the mixing is roughly equivalent to computing the HOSVD
initialization. 

Stopping conditions present an interesting dilemma for any randomized
method. 
Since each subproblem is now solved inexactly, 
the fit is no longer monotonically increasing.
In particular, it is difficult to detect when improvement has
stagnated. 
If we know the amount of noise in advance, we can terminate once the
desired fit is achieved. However, this assumes not only that we know
the noise but also that our model is good in the sense that
the data has inherent multilinear structure and the rank is known. 
Instead,
we propose a modification of the standard stagnation metric: track the
best fit and stop when it fails to improve for more than, say, ten
iterations. 

Although there is some theory on the number of samples required for least squares
as in \cref{eqn:numsamples}, they are impractical for most
implementations. 
We lack a rigorous way of estimating a good sample size.
In practice, we have found that a small multiple of
$R$ is sufficient, i.e. 10--100 times $R$.
Clearly, more theoretical work to justify this choice is needed.

Since the CP fitting problem is non-convex, CP-ALS cannot
guarantee global optimality. This is unchanged for randomized
methods. Moreover, as mentioned above, we do not even have a guarantee
of improving the objective function at each step. However, our
experimental results indicate that randomized methods are more robust
to the starting point, so this is a potential advantage and perhaps a
topic for future research.

Many lines of future research remain, in addition to those mentioned
above.
As discussed in \cref{sec:rcpals-cost}, the most expensive part of
CPRAND is extracting the random fibers from the dense tensor (i.e.,
memory operations), so we may consider both algorithmic adjustments or
specialized implementations to alleviate that expense.
We would also like to compare directly to other improved methods for computing CP, such as those in \cite{phan,VeLa16}.
Another topic of investigation is to prove that the mixing operator we develop in
\cref{sec:cprand-mix} is  an FJLT.
An obvious extension is consideration of sparse tensors. In the sparse
case, we never form the Khatri-Rao matrix (see \cite{stop1}), so the
bottlenecks are different.
We note that our mixing process would convert the sparse tensor to a dense one, so we suspect a non-uniform sampling scheme without mixing (as in \cite{spals}) will be more effective for sparse data.
These sketching methods also naturally extend to out-of-core 
algorithms, and may even be used to increase scalability in
distributed memory.  
Finally, it is natural to consider application of 
randomization to other decompositions such as Tucker \cite{Tu66}, tensor train \cite{Os11},
or functional tensor decompositions \cite{ChLeNoRa15,GoKaMa15}.

\section*{Acknowledgment}
We thank Nico Vervliet for generously sharing his scripts to
preprocess the hazardous gases experiment data used in \cref{sec:hazardous-gases}.
We would like to thank Alex Williams for the CP decomposition visualization script used to create \cref{fig:gases}.
This material is based upon work supported by the Sandia Truman Postdoctoral Fellowship and the U.S. Department of Energy, Office of Science, Office of Advanced Scientific Computing Research, Applied Mathematics program.
Sandia National Laboratories is a multi-mission laboratory managed and
operated by Sandia Corporation, a wholly owned subsidiary of Lockheed
Martin Corporation, for the U.S. Department of Energy's National
Nuclear Security Administration under contract DE--AC04--94AL85000.
\bibliographystyle{siamplain}
\bibliography{tensbib}

\end{document}